\documentclass[a4paper,11pt]{article}
\usepackage[T1]{fontenc}
\usepackage{amsfonts}
\usepackage{verbatim}
\usepackage{amsmath, amsthm, amssymb, mathrsfs}
\usepackage{graphicx}

\usepackage[margin=1.35in]{geometry}

\title{On Approximations of Subordinators in $L^p$ and the Simulation of Tempered Stable Distributions}
\author{Michael Grabchak\footnote{Email address: mgrabcha@charlotte.edu}\ \ and Sina Saba\footnote{Email address: ssaba3@charlotte.edu}\\
{\it University of North Carolina Charlotte}}

\begin{document}
\newtheorem{prop}{Proposition}
\newtheorem{thrm}{Theorem}
\newtheorem{defn}{Definition}
\newtheorem{cor}{Corollary}
\newtheorem{lemma}{Lemma}
\newtheorem{remark}{Remark}
\newtheorem{example}{Example}

\newcommand{\rd}{\mathrm d}
\newcommand{\rE}{\mathrm E}
\newcommand{\TS}{\mathrm{TS}_\alpha}
\newcommand{\tr}{\mathrm{tr}}
\newcommand{\iid}{\stackrel{\mathrm{iid}}{\sim}}
\newcommand{\eqd}{\stackrel{d}{=}}
\newcommand{\approxd}{\stackrel{d}{\approx}}
\newcommand{\cond}{\stackrel{d}{\rightarrow}}
\newcommand{\conv}{\stackrel{v}{\rightarrow}}
\newcommand{\conw}{\stackrel{w}{\rightarrow}}
\newcommand{\conp}{\stackrel{p}{\rightarrow}}
\newcommand{\confdd}{\stackrel{fdd}{\rightarrow}}
\newcommand{\plim}{\mathop{\mathrm{p\mbox{-}lim}}}
\newcommand{\lgg}{\mathrm{log}}
\newcommand{\dlim}{\operatorname*{d-lim}}
\newcommand{\ID}{\mathrm{ID}}

\newcommand{\esssup}{\operatorname*{ess\,sup}}

\newcommand{\R}{\mathbb{R}}
\newcommand{\N}{\mathbb{N}}
\newcommand{\C}{\mathbb{C}}
\newcommand{\Q}{\mathbb{Q}}
\newcommand{\PB}{\mathbb{P}}
\renewcommand{\angle}[2]{\langle #1,#2 \rangle}
\newcommand{\norm}[1]{\left\| #1 \right\|}
\newcommand{\DTS}{\mathrm{DTS}}
\newcommand{\PTS}{\mathrm{PTS}}
\newcommand{\Pois}{\mathrm{Pois}}
\newcommand{\LL}{\mathrm{LL}}
\newcommand{\Ga}{\mathrm{Ga}}
\newcommand{\GGa}{\mathrm{GGa}}
\newcommand{\U}{\mathrm{U}}
\newcommand{\B}{\mathrm{B}}
\newcommand{\Levy}{L\' evy }

\maketitle

\begin{abstract}
Subordinators are infinitely divisible distributions on the positive half-line.  They are often used as mixing distributions in Poisson mixtures. We show that appropriately scaled Poisson mixtures can approximate the mixing subordinator and we derive a rate of convergence in $L^p$ for each $p\in[1,\infty]$. This includes the Kolmogorov and Wasserstein metrics as special cases.  As an application, we develop an approach for approximate simulation of the underlying subordinator. In the interest of generality, we present our results in the context of more general mixtures, specifically those that can be represented as differences of randomly stopped L\'evy processes. Particular focus is given to the case where the subordinator belongs to the class of tempered stable distributions.\\

\end{abstract}

\section{Introduction}

Subordinators are a wide class of positive random variables. They correspond to the class of infinitely divisible distributions on the positive half-line. See \cite{Sato:1999} or \cite{Steutel:vanHarn:2004} for many properties and examples. Subordinators have many uses and applications. They are often used as mixing distributions in the context of Poisson mixtures, which are common models for over-dispersed count data. We show that appropriately scaled Poisson mixtures can approximate the mixing subordinator and we derive a rate of convergence in $L^p$ for each $p\in[1,\infty]$. When $p=\infty$ this is the Kolmogorov metric and when $p=1$ it is the Kantorovich or  Wasserstein metric. As an application, we develop a methodology for approximate simulation of the underlying subordinator.
In the interest of generality, we present our theoretical results in the context of more general mixtures. Specifically, those that can be represented as differences of randomly stopped L\'evy processes.

We are particularly interested by the case where the mixing distribution is tempered stable (TS). TS distributions form a large class of models, which modify the tails of infinite variance stable distributions to make them lighter. This leads to more realistic models for many application areas. We are motivated by applications to finance, where TS subordinators are common models for stochastic volatility (\cite{Barndorff-Nielsen:Shephard:2001}, \cite{Barndorff-Nielsen:Shephard:2003}) and their bilateral versions are often used to model financial returns (\cite{Kuchler:Tappe:2013}, \cite{Kuchler:Tappe:2014}). Tweedie distributions introduced in \cite{Tweedie:1984} are, perhaps, the earliest class of TS distributions to be studied. More general classes were introduced in \cite{Rosinski:2007} and \cite{Rosinski:Sinclair:2010}. See also the monograph \cite{Grabchak:2016} and the references therein.  Poisson mixtures, where the mixing distribution is TS are called discrete TS distributions, see \cite{Grabchak:2018} and \cite{Grabchak-DTS}. In the special case where the mixing distribution belongs to the Tweedie class, they are called Poisson-Tweedie. See \cite{Baccini:Barabesi:Stracqualursi:2016} for a recent overview.

Formally, a Poisson mixture can be described as follows. Let  $X$ be a subordinator and, independent of $X$, let $\{Z(t):t\ge0\}$ be a Poisson process with rate $1$. The Poisson mixture is then given by $Z(X)$. For a survey of Poisson mixtures see \cite{Karlis:Xekalaki:2005}  or Section VI.7 of \cite{Steutel:vanHarn:2004}. We consider a more general bilateral situation, where $X_+,X_-$ are independent subordinators and, independent of these, $Z^+=\{Z^+(t):t\ge0\}$ and $Z^-=\{Z^-(t):t\ge0\}$ are independent L\'evy processes. Later we will allow these processes to be quite general, but, for the moment, assume that they are independent Poisson processes each with rate $1$ and set
\begin{eqnarray}\label{eq: bilat poisson}
a Z^+(X^+/a)-  a Z^-(X^-/a).
\end{eqnarray}
This gives a random variable whose support is concentrated on the lattice $a\mathbb Z=\{0,\pm a,\pm2a,\pm3a,\dots\}$.  

Distributions of this type are important in the context of high frequency finance. In practice, trades on an exchange cannot be for arbitrary amounts. There is a smallest price increment called the `tick size,' and all prices changes must be in the form of an integer times the tick size. Thus, if $a$ is the tick size, then price changes must take values on the lattice $a\mathbb Z$. For this reason, distributions of the form \eqref{eq: bilat poisson} are often used to model high frequency price changes, see \cite{Barndorff-Nielsen:Pollard:Shephard:2012} or \cite{Koopman:Lit:Lucas:2017}. On the other hand, less frequent price changes are typically modeled using continuous distributions, often bilateral subordinators. The idea is that over larger time frames, the tick size does not matter (so long as it is relatively small). We explain the transition from discrete to continuous models, by showing that
\begin{eqnarray}\label{eq: main conv intro}
a Z^+(X^+/a)-  a Z^-(X^-/a) \cond X^+-X^- \mbox{ as } a\downarrow0.
\end{eqnarray}
Further, this result gives an approximate simulation method for simulating subordinators. The ability to simulate subordinators is particularly important as these distributions form the building blocks from which other distributions, including many multivariate distributions can be simulated, see \cite{Xia:Grabchak:2022}.

The purpose of this paper is two-folds. First, we prove that \eqref{eq: main conv intro} holds and we derive a rate of convergence in $L^p$ for each $p\in[1,\infty]$. In the interest of generality, we do not require $Z^+$ and $Z^-$ to be Poisson processes, but allow them to be more general L\'evy processes. Second, we use this convergence to justify an approximate simulation method. We show that one can approximately simulate $ X^+-X^-$ by simulating $a Z^+(X^+/a)-  a Z^-(X^-/a)$ for a relatively small $a$ and we derive approaches for simulating the latter. This is particularly important for TS distributions, where, with the exception of some relatively simple special cases (\cite{Hofert:2011}, \cite{Grabchak:2021}), no exact simulation methods are known.

The rest of this paper is organized as follows. In Section \ref{sec: background} we recall some basic facts about infinitely divisible distributions and subordinators and in Section \ref{sec: main results} we give our main theoretical results. In Section \ref{sec: sub Pois} we discuss Poisson mixtures more broadly and give our methodology for approximate simulation of the mixing distribution. In Section \ref{sec: sims} we give a small scale simulation study. Proofs are postponed to Section \ref{sec: proofs}. In Section \ref{sec: extensions} we consider extension of our results. These include an application to normal variance models, which are popular in financial applications.

Before proceeding we introduce some notation. We use the standard abbreviations: iid for independent and identically distributed, cdf for cumulative distribution function, pdf for probability density  function, and pmf for probability mass function. We write $\mathbb R$ to denote the set of real numbers and we write $\mathfrak B(\mathbb R)$ to denote the class of Borel sets on $\mathbb R$. We write $\mathbb C$ to denote the set of complex numbers and, for $z\in\mathbb C$, we write $\Re z$ and $\Im z$ to denote the real and imaginary parts of $z$, respectively. We write $\wedge$ and $\vee$ to denote the minimum and maximum, respectively.  We write $\mathrm{Pois}(\lambda)$, $U(a,b)$, $N(0,1)$, and $\delta_a$ to denote, respectively, the Poison distribution with mean $\lambda$, the uniform distribution on $(a,b)$, the standard normal distribution, and the point mass at $a$. Further, we write  $\Ga(\beta,\zeta)$ to denote a gamma distribution with pdf 
$$
f_{\beta,\zeta}(x)=\frac{\zeta^{\beta}}{\Gamma(\beta)}x^{\beta-1}e^{-\zeta x},\ x>0,
$$
where $\beta,\zeta>0$ are parameters. For a set $A$ we write $1_A$ to denote the indicator function on $A$ and we write $:=$ to denote a defining equality. If $\mu$ is a probability distribution, we write $X\sim \mu$ to denote that $X$ is a random variable with distribution $\mu$ and we write $X_1,X_2,\dots\iid\mu$ to denote that $X_1,X_2,\dots$ are iid random variables with common distribution $\mu$.

\section{Background}\label{sec: background}

In this paper we focus on infinitely divisible distributions with finite variation and no drift. A distribution $\mu$ of this type has a characteristic function of the form
$$
\hat\mu(s) =e^{C_\mu(s)}, \ \ s\in\mathbb R,
$$
where
$$
C_\mu(s)=\int_{-\infty}^\infty \left(e^{ixs}-1\right) M(\rd x), \ \ s\in\mathbb R
$$
is the cumulant generating function and $M$ is the L\'evy measure satisfying
\begin{eqnarray}\label{eq: levy meas}
M(\{0\}) = 0 \mbox{ and } \int_{-\infty}^\infty \left(|x|\wedge1\right)M(\rd x)<\infty.
\end{eqnarray}
We write $\mu=\ID_0(M)$ to denote this distribution. For future reference, we note that
\begin{eqnarray}\label{eq: real part c is neg}
\Re C_\mu(s)=\int_{-\infty}^\infty \left(\cos(xs)-1\right) M(\rd x)\le0, \ \ s\in\mathbb R.
\end{eqnarray}
Associated with every infinitely divisible distribution $\mu$ is a L\'evy process $\{Z(t):t\ge0\}$, where $Z(1)\sim\mu$. This process has independent and stationary increments. An important example is when $M(\rd x) = \lambda \delta_1(\rd x)$. In this case, $\ID_0(M)$ is a Poisson distribution with mean $\lambda$ and the associated L\'evy process is a Poisson process with rate $\lambda$. See \cite{Sato:1999} for many results about infinitely divisible distributions and their associated L\'evy processes.

When $M((-\infty,0])=0$, the distribution $\mu=\ID_0(M)$ is called a subordinator. By a slight abuse of terminology, we also use this term to describe any random variable from this distribution. In this case, $\mu((-\infty,0))=0$ and the associated L\'evy process is increasing. When $M((0,\infty))=0$, the distribution $\mu=\ID_0(M)$ satisfies $\mu((0,\infty))=0$ and the associated L\'evy process is decreasing. In general, under our assumptions, we can decompose the L\'evy measure $M$ into two parts: 
\begin{eqnarray}\label{eq: M +}
M^+(B) = M(B\cap(0,\infty)),\ \ B\in \mathfrak B(\mathbb R)
\end{eqnarray}
and 
\begin{eqnarray}\label{eq: M -}
M^-(B) = M((-B)\cap(-\infty,0)), \ \ B\in \mathfrak B(\mathbb R).
\end{eqnarray}
In this case, if $X^-\sim\ID_0(M^-)$ and $X^+\sim\ID_0(M^+)$ are independent and $X=X^+-X^-$, then $X\sim\ID_0(M)$ and $X^-,X^+$ are subordinators. Since $X$ is the difference of two independent subordinators, we sometimes call it (or equivalently its distribution) a bilateral subordinator.

TS distributions form an important class of  infinitely divisible distributions. They correspond to the case, where the L\'evy measure is of the form
\begin{eqnarray}\label{eq: TS Levy meas}
M(\rd x) = \eta_- g(x) x^{-1-\alpha} 1_{[x<0]}\rd x+\eta_+ g(x) x^{-1-\alpha} 1_{[x>0]}\rd x,
\end{eqnarray}
where $\eta_{\pm}\ge0$ with $\eta_++\eta_->0$, $\alpha\in(0,1)$, and $g:\mathbb R\mapsto[0,\infty)$ is a bounded Borel function satisfying
$$
\lim_{x\to0}g(x) = 1.
$$
We call $g$ the tempering function and we denote the distribution $\mu=\ID_0(M)$ by $\mu=\TS(g,\eta_-,\eta_+)$. We note that one can consider the more general case where $\alpha\in(0,2)$. However, when $\alpha\in[1,2)$, the corresponding distributions do not have finite variation and are, thus, beyond the scope of this paper. 

When $g\equiv1$, the distribution $\TS(g,\eta_-,\eta_+)$ reduces to an infinite variance stable distribution, see \cite{Samorodnitsky:Taqqu:1994}. It is often assumed that $\lim_{x\to\pm\infty}g(x)=0$, which leads to distributions that are similar to stable distributions in some central region, but with lighter (i.e.,\ tempered) tails. This explains the name and makes these distributions useful for many application areas, see  \cite{Grabchak:Samorodnitsky:2010} for a discussion in the context of finance. Many facts about TS distributions can be found in \cite{Rosinski:2007}, \cite{Rosinski:Sinclair:2010}, \cite{Grabchak:2016}, and the references therein.

\section{Main Results}\label{sec: main results}

Our main theoretical results are concerned with randomly stopped L\'evy processes.  For this reason, we generally work with two infinitely divisible distributions: $\mu=\ID_0(M)$ and $\nu=\ID_0(L)$. For us, $\mu$ is the distribution of the random time when the process is stopped and $\nu$ is the distribution of the L\'evy process. Throughout, we take
$$
X\sim \mu=\ID_0(M)\mbox{ and } Z\sim \nu=\ID_0(L)
$$
to denote generic random variables  from distributions $\mu$ and $\nu$ respectively. At this point, we make no additional assumptions on $\mu$, but for distribution $\nu$, we assume, throughout, that
$$
\zeta_1:=\int_{-\infty}^\infty |z|L(\rd z)<\infty\mbox{ and } \gamma:=\int_{-\infty}^\infty z L(\rd z)>0.
$$
It can be shown that $\gamma=\rE[Z]$. 

We decompose $M$  into $M^+$ and $M^-$ as in \eqref{eq: M +} and \eqref{eq: M -}. Let $X^-\sim\ID_0(M^-)$ and $X^+\sim\ID_0(M^+)$ be independent random variables and note that $X\eqd X^+-X^-$. Independent of these, let $Z^+=\{Z^+(t):t\ge0\}$ and $Z^-=\{Z^-(t):t\ge0\}$ be independent L\'evy processes with $Z^+(1),Z^-(1)\sim\nu$.  For $a>0$ set 
\begin{eqnarray*}
Y_a = a Z^+(X^+/(a\gamma))-  a Z^-(X^-/(a\gamma))
\end{eqnarray*}
and let $\mu_a$ be the distribution of $Y_a$. A simple conditioning argument shows that the characteristic function of $Y_a$ is given by
\begin{eqnarray}\label{eq: mu a}
\hat\mu_a(s) = \rE\left[e^{isY_a}\right]= \exp\left\{\int_{-\infty}^\infty \left(e^{\frac{|x|}{a\gamma} C_\nu\left(sa\frac{x}{|x|}\right)}-1\right) M(\rd x)\right\}, \ \ \ s\in\mathbb R,
\end{eqnarray}
where $C_\nu$ is the cumulant generating function of $\nu$.

\begin{prop}\label{prop: conv}
We have
$$
Y_a\cond X \ \mbox{ as }\ a\downarrow0.
$$
\end{prop}

In the case where $Z^-$ and $Z^+$ are Poisson processes, versions of this result can be found in \cite{Klebanov:Slamova:2013}, \cite{Grabchak:2018}, and Proposition 6.5 of \cite{Steutel:vanHarn:2004}. Our goal is to characterize the rate of convergence in $L^p$. Let $F$ be the cdf of distribution $\mu$ and let $F_a$ be the cdf of distribution $\mu_a$. Recall that, for $p\in[1,\infty)$, the $L^p$ distance between these distributions is given by
$$
\|F-F_a\|_p = \left(\int_{-\infty}^\infty \left|F(x)-F_a(x)\right|^p\rd x\right)^{1/p} 
$$
and for $p=\infty$ it is given by
$$
\|F-F_a\|_\infty = \sup_{x\in\mathbb R} \left|F(x)-F_a(x)\right|. 
$$
Since $ \left|F(x)-F_a(x)\right|\le1$ for each $x\in\mathbb R$, it is readily checked that for $1\le p_1\le p_2<\infty$ we have
$$
\|F-F_a\|_{p_2}^{p_2} \le \|F-F_a\|^{p_1}_{p_1}.
$$
In this context, the $L^\infty$ distance is sometimes called the Kolmogorov metric and the $L^1$ distance is sometimes called the Kantorovich or the Wasserstein metric. Many results about these and related distances can be found in \cite{Gibbs:Su:2002}, \cite{Bobkov:2016}, \cite{Arras:Houdre:2019}, and the references therein. We just note the following equivalent formulations of $L^1$ distance. We have
$$
\|F-F_a\|_1 = \inf \{\rE\left[|X-Y_a|\right]\} = \sup_h \left\{\left|\rE\left[h(X)\right]-\rE\left[h(Y_a)\right]\right|\right\},
$$
where the infimum is taken over all possible joint distributions of $X$ and $Y_a$ and the supremum is taken over all continuous functions $h$ satisfying the Lipschitz condition $|h(x)-h(y)|\le |x-y|$ for $x,y\in\mathbb R$.

Before stating our main results, we introduce several quantities. We do not, in general, assume these to be finite. Let 
$$
r_0 := \int_{\mathbb R} |\hat\mu(s)|\rd s, \ m_1 := \int_{-\infty}^\infty |x| M(\rd x) ,\  m_2 := \int_{-\infty}^\infty x^2 M(\rd x),
$$
and let 
$$
 \zeta_2 := \int_{-\infty}^\infty z^2 L(\rd z).
$$
It is well-known that if $r_0<\infty$, then $\mu$ has a pdf $f$ satisfying
$$
\esssup_{x\in\mathbb R} f(x)\le \frac{r_0}{2\pi}.
$$
Further, by Corollary 25.8 in \cite{Sato:1999} we have
$$
m_i<\infty \mbox{ if and only if } \rE[|X|^i]<\infty , \ \ i=1,2
$$
and it is readily checked that, when they are finite, $m_1\ge\rE[|X|]$, $m_2=\mathrm{Var}(X)$, and $\zeta_2=\mathrm{Var}(Z)$. We now give our main result.

\begin{thrm}\label{thrm: main gen}
Assume that $r_0,m_1,\zeta_1,\zeta_2$ are all finite.  For any $a>0$, 
$$
\left\| F_{a} - F \right\|_\infty \le  \sqrt a\left( \frac{1 }{\pi}e^{0.5m_1\zeta_2/\gamma} m_1\frac{\zeta_2}{2\gamma}+ \frac{12}{\pi^2}\right)  r_0= \mathcal O(a^{1/2})
$$
and for any $p\in[2,\infty)$ and any $a>0$
$$
\left\| F_{a} - F \right\|_p \le \left( \sqrt a e^{0.5m_1\zeta_2/\gamma} m_1\frac{\zeta_2}{4\gamma} r_0^{1-1/p}   +  a^{1/(2p)}4 (p-1)\right) = \mathcal O\left(a^{1/(2p)}\right).
$$
If, in addition, $m_2<\infty$, then for any  $p\in[1,\infty)$ and any $a>0$
\begin{eqnarray*}
\left\| F_{a} - F \right\|_p &\le& \left\| F_{a} - F \right\|_1^{1/p} \le C_a^{1/p} a^{1/(2p)}=  \mathcal O\left(a^{1/(2p)}\right),
\end{eqnarray*}
where
\begin{eqnarray*}
C_a =  \left(\left(\sqrt a +1\right)  e^{0.5m_1\zeta_2/\gamma} m_1+  \frac{\zeta_1}{\gamma} m_2 + 2\sqrt a m_1 +  e^{0.5 m_1 \zeta_2/\gamma} m_1^2\frac{\zeta_1}{\gamma} \right)\frac{\zeta_2}{2\gamma} r_0^{1/2} + 4 \pi .
\end{eqnarray*}
\end{thrm}

It is well-known that the $L^1$ and $L^\infty$ distances provide bounds on a number of other distances, including the L\'evy, discrepancy, and Prokhorov metrics, see \cite{Gibbs:Su:2002}. Thus, our bounds give rates of convergence for these distances as well. We now give a better rate of convergence for the case where $\mu$ has a TS distribution.

\begin{thrm}\label{thrm: main TS}
Let $\mu=\TS(g,\eta_-,\eta_+)$ and assume that $m_1,\zeta_1,\zeta_2$ are all finite.  For any $a>0$, 
$$
\left\| F_{a} - F \right\|_\infty \le \mathcal O\left( a^{\frac{1}{2-\alpha}} \right)
$$
and for any $p\in[2,\infty)$ and any $a>0$
$$
\left\| F_{a} - F \right\|_p \le  \mathcal O\left(a^{\frac{1}{p(2-\alpha)}} \right).
$$
If, in addition, $m_2<\infty$, then for any $p\in[1,\infty)$ and any $a>0$
\begin{eqnarray*}
\left\| F_{a} - F \right\|_p \le \left\| F_{a} - F \right\|_1^{1/p} \le \mathcal O\left( a^{\frac{1}{p(2-\alpha)}} \right).
\end{eqnarray*}
\end{thrm}

In the proof we will see that we always have $r_0<\infty$ in this case. 

\section{Poisson Mixtures and Simulation}\label{sec: sub Pois}

Our main results imply that $\mu_a \approx \mu$ when $a$ is small, which suggests that we can approximately simulate from $\mu$ by simulating from $\mu_a$. In practice, simulation from $\mu_a$ is often easier than simulation from $\mu$. This is particularly true when the L\'evy processes $Z^+,Z^-$ are Poisson processes. In this case, the distribution of $\mu_a$ is a scaled Poisson mixture. A general review of Poisson mixtures can be found in \cite{Karlis:Xekalaki:2005} or Section VI.7 of \cite{Steutel:vanHarn:2004}. Our discussion primarily extends results about discrete TS distributions given in \cite{Grabchak-DTS}. 

Throughout this section we assume that $M((-\infty,0])=0$, which means that $\mu=\ID_0(M)$ is a subordinator. Further, we take $L(\rd z) = \delta_1(\rd z)$. This means that $\nu=\ID_0(L)=\Pois(1)$ and thus that the L\'evy process $Z^+=\{Z^+(t):t\ge0\}$ with $Z^+(1)\sim \nu$ is a Poisson process with rate $1$ and that $\gamma=1$. Let $X\sim \ID_0(M)$ be independent of $Z^+$ and, for any $a>0$, set
$$
Y_a = a Z^+(X/a).
$$
As before, the distribution of $Y_a$ is denoted by $\mu_a$. In this case, we sometimes write $\mu_a=\mathrm{PM}(M,a)$. Note that the support of this distribution is concentrated on $\{0,a,2a,3a,\dots\}$.

Specializing \eqref{eq: mu a} to the current situation shows that the characteristic function of $Y_a$ is given by
$$
\hat\mu_a(z) = \rE\left[e^{izY_a}\right]= \exp\left\{\int_{0}^\infty \left(e^{x (e^{iza}-1)/a}-1\right) M(\rd x)\right\}, \ \ \ z\in\mathbb R.
$$
Next, we introduce several useful quantities. These are closely related to the canonical sequence of discrete infinitely divisible distributions discussed in \cite{Steutel:vanHarn:2004}. For $a>0$, let
$$
\ell_k^{(a)} = \int_{0}^{\infty} e^{-x/a} (x/a)^{k} M(\rd x), \quad k = 1,2, \ldots
$$
and let 
$$
\ell_+^{(a)} = \sum_{k=1}^{\infty} \frac{\ell_k^{(a)}}{k!} = \int_{0}^{\infty} (1-e^{-x/a}) M(\rd x).
$$
We now characterize the distribution $\mu_a$ in terms of these quantities.

\begin{prop} \label{prop: facts mixed Pois}
\begin{enumerate}
    \item The characteristic function of $\mu_a$ is given by
    $$
  \hat\mu_a(s) =   \exp \left\{\sum_{k=1}^{\infty}  \frac{\ell_k^{(a)}}{k!}\left(e^{isa k} -1\right)\right\} = e^{-\ell_+^{(a)}} \exp \left\{\sum_{k=1}^{\infty}  \frac{\ell_k^{(a)}}{k!}e^{isa k} \right\}, \ \ s \in \mathbb R.
    $$   
   \item  We have $\mu_a=\ID_0(M^{(a)})$, where the \Levy measure $M^{(a)}$ is concentrated on $\{a,2a,\ldots\}$ with
    $$
    M^{(a)}(\{a k\}) = \frac{1}{k!} \ell_k^{(a)}, \quad k = 1,2,\ldots,
    $$
    and $M^{(a)} ((0, \infty)) = \ell_+^{(a)} < \infty$.
     \item Let $p_k^{(a)} = P(Y_a =  ak)$ for $k = 0, 1, 2, \ldots$. We have $p_0 = e^{- \ell_+^{(a)}}$ and
    $$
    p_k^{(a)} = \frac{1}{k} \sum_{j=0}^{k-1} \frac{\ell_{k-j}^{(a)}}{(k-j-1)!} p_j^{(a)}, \quad k = 1,2,\ldots  .
    $$
    \item Let $M_1^{(a)}$ be the probability measure on $\{a,2a, \ldots\}$ with 
    $$
    M_1^{(a)}(\{ak\}) = \frac{M^{(a)}(\{ak\})}{M^{(a)}((0, \infty))} = \frac{\ell_k^{(a)}}{k! \ell_+^{(a)}}, \quad k = 1,2,\ldots.
    $$
    If $N \sim \Pois\left(\ell_+^{(a)}\right)$ and $W_1, W_2, \ldots \stackrel{iid}{\sim} M_1^{(a)}$ are independent of $N$, then
    $$
    Y_a\eqd \sum_{i=1}^{N} W_i .
    $$
\end{enumerate}
\end{prop} 

Here, $M^{(a)}$ is a discretization of $M$. Parts 3 and 4 can be used to simulate from $\mu_a$. First we describe the algorithm suggested by Part 3. It is a version of the inverse transform method.\\

\noindent{\bf Algorithm 1}: Simulation from $\mu_a=\mathrm{PM}(M,a)$.\\
\textbf{1.} Simulate $U \sim U(0,1)$.\\
\textbf{2.} Evaluate $Y = \min\left\{m=0,1,2, \ldots | \sum_{k=0}^m p^{(a)}_k>U\right\}$ and return $aY$.\\

This method often works well. However, for heavier tailed distributions, we may need to calculate $p^{(a)}_k$ for a large value of $k$, which can be inefficient. An alternate simulation method is suggested by Part 4. \\

\noindent{\bf Algorithm 2}: Simulation from $\mu_a=\mathrm{PM}(M,a)$.\\
\textbf{1.} Simulate $N \sim \Pois(\ell_+^{(a)})$.\\
\textbf{2.} Simulate $W_1, W_2, \ldots, W_N \stackrel{iid}{\sim} M_1^{(a)}$.\\
\textbf{3.} Return $\sum_{i=1}^{N} W_i$.  \\

To use Algorithm 2 we need a way to simulate from $M_1^{(a)}$. Following ideas in \cite{Grabchak-DTS} for discrete TS distributions, we can represent  $M_1^{(a)}$ as a mixture of scaled zero-truncated Poisson distributions. Let $\Pois_+(\lambda)$ denote the zero truncated Poisson distribution with parameter $\lambda$ and pmf 
$$
p_{\lambda}(k) = \frac{\lambda^{k} e^{- \lambda}}{k! (1-e^{-\lambda})}, \ \ k=1,2,\dots.
$$
We have
\begin{align*}
M_1^{(a)}(\{ak\}) &= \frac{1}{\ell_+^{(a)}} \int_{0}^{\infty} \frac{(x/a)^k e^{-x/a}}{k! (1-e^{-x/a})} (1-e^{-x/a}) M(\rd x) \\
           &= \int_{0}^{\infty} p_{x/a}(k) M_2^{(a)}(\rd x), 
\end{align*}
where 
$$
M_2^{(a)}(\rd x)= \frac{1}{\ell_+^{(a)}} (1-e^{-x/a}) M(\rd x)
$$
is a probability measure. Thus, we can simulate from $M_1^{(a)}$ as follows.\\

\noindent{\bf Algorithm 3}: Simulation from $M_1^{(a)}$.\\
\textbf{1.} Simulate $\lambda \sim M_2^{(a)}$.\\
\textbf{2.} Simulate $W \sim \Pois_+(\lambda)$ and return $aW$.\\

In this paper, to simulate from $\Pois_+(\lambda)$, we use {\em rztpois} method in the actuar package for the statistical software R, see \cite{Dutang}. Simulation from $M_2^{(a)}$ depends on the structure of $M$. For certain discrete TS distributions and $a=1$ two methods were given in \cite{Grabchak-DTS}. We now extend them to the case of general $a>0$.  For TS distributions, $M$ is of the form \eqref{eq: TS Levy meas} and, to ensure $M((-\infty,0])=0$, we take $\eta_-=0$. It follows that, in this case, 
$$
M_2^{(a)}(\rd x)= \frac{\eta_+}{\ell_+^{(a)}} (1-e^{-x/a}) g(x)x^{-1-\alpha}1_{[x>0]}\rd x.
$$
Let
$$
M_3^{(a)}(\rd x)= \frac{1}{K_a} (1-e^{-x}) g(xa)x^{-1-\alpha} 1_{[x>0]} \rd x,
$$
where $K_a=\int_0^\infty (1-e^{-x}) q(xa)x^{-1-\alpha} \rd x= a^\alpha\ell_+^{(a)}/\eta_+$ is a normalizing constant. It is easily checked that, if $Y\sim M_3^{(a)}$, then $aY\sim M_2^{(a)}$. 
Two rejection sampling methods for simulating from $M_3^{(a)}$ are given in \cite{Grabchak-DTS}.  Toward this end, let $C_g=\esssup_{x>0} g(x)$ and set
$$
\varphi_1(x) = C_g^{-1}\frac{g(xa)(1-e^{-x})}{x\wedge1}, \ \ x>0.
$$

\noindent \textbf{Algorithm 4.} Simulation from $M_2^{(a)}$ for TS distributions.\\
\textbf{1.} Independently simulate $U,U_1,U_2\iid U(0,1)$ and set $V=U_1^{1/(1-\alpha)}U_2^{-1/\alpha}$.\\
\textbf{2.} If $U\le \varphi_1(V)$ return $aV$, otherwise go back to step 1.\\

On a given iteration, the probability of acceptance is $K_a \alpha(1-\alpha)/C_g$. In some cases, we can increase the acceptance rate by using a different trial distribution. The second method works when there exist $p,C,\zeta>0$ such that, for Lebesgue a.e.\ $x$,
\begin{eqnarray}\label{eq: cond on q}
g(x) \le C e^{-\zeta x^p}.
\end{eqnarray}
In this case let
$$
\varphi_2(x) = \frac{g(x)(1-e^{-x})}{Cxe^{-\zeta x^p}} ,\ x>0.
$$
We get the following rejection sampling algorithm, where, on a given iteration, the probability of acceptance is 
$
\frac{K_a p\zeta^\frac{1-\alpha}{p}}{C\Gamma\left(\frac{1-\alpha}{p}\right)}
$.\\

\noindent \textbf{Algorithm 5.} Simulation from $M_2^{(a)}$ for TS distributions when \eqref{eq: cond on q} holds.\\
\textbf{1.} Independently simulate $U\sim U(0,1)$ and $V\sim \Ga((1-\alpha)/p,\zeta)$.\\
\textbf{2.} If $U\le \varphi_2(V^{1/p})$ return $aV^{1/p}$, otherwise go back to step 1.\\

Note that, in this section, the assumption that $M((-\infty,0])=0$ is without loss of generality. When it does not hold, we can decompose $M$ into $M^+$ and $M^-$ as in \eqref{eq: M +} and \eqref{eq: M -}. We can then independently simulate $Y_{a,+}\sim \mathrm{PM}(M^+,a)$ and $Y_{a,-}\sim \mathrm{PM}(M^-,a)$ and take 
\begin{eqnarray*}
Y_a = Y_{a,+}-Y_{a,-}.
\end{eqnarray*}

\section{Simulations}\label{sec: sims}

In this section we give a small simulation study to illustrate our methodology for approximate simulation from a subordinator $\mu$ by simulating from $\mu_a$ with some small $a>0$. We focus on two classes of TS distributions: the class of classical tempered stable (CTS) distributions and the class of TS distributions with power tempering (PT distributions). CTS distributions are, arguably, the best known and most commonly used TS distributions, see \cite{Kuchler:Tappe:2013} for a survey. While exact simulation methods for CTS distributions are well-known, see e.g.\ \cite{Kawai:Masuda:2011} or \cite{Hofert:2011}, we use them to illustrate our methodology because they are so popular. On the other hand, the only known simulation methods for PT distributions are either approximate or computationally intensive, see \cite{Grabchak:2019}. For a discussion of PT distributions and their applications, see \cite{Grabchak:2016} or \cite{Grabchak:2021b}. 
In the symmetric case, methods to numerically evaluate pdfs, cdfs, and quantiles of CTS and PT distributions are given in the SymTS package for the statistical software R.

We begin with the class of CTS distributions. Distribution $\TS(\eta_-,\eta_+,g)$ is CTS if
$$
g(x) = e^{-|x|/\ell_-}1_{[x\le0]} + e^{-x/\ell_+} 1_{[x\ge0]}
$$
for some $\ell_-,\ell_+>0$. It is often convenient to take $\eta_-=c_-\ell_-^\alpha$ and $\eta_+=c_+\ell_+^\alpha$, where $c_-,c_+\ge0$ with $c_-+c_+>0$ are parameters. In this case  
we denote this distribution by CTS$_\alpha(c_-,\ell_-,c_+,\ell_+)$. When $c_-=0$ we get an important subset of the class of Tweedie distributions, see \cite{Tweedie:1984}.

Our methodology for approximate simulation from $\mu=$CTS$_\alpha(c_-,\ell_-,c_+,\ell_+)$ is as follows. First choose $a>0$ small, then independently simulate $Y_{a,+}\sim\mathrm{PM}(M^+,a)$ and $Y_{a,-}\sim\mathrm{PM}(M^-,a)$, where $M^+(\rd x) = c_+\ell_+^\alpha e^{-x/\ell_+}x^{-1-\alpha} 1_{[x>0]}\rd x$ and $M^-(\rd x) = c_-\ell_-^\alpha e^{-x/\ell_-}x^{-1-\alpha} 1_{[x>0]}\rd x$. The random variable $Y_a = Y_{a,+}-Y_{a,-}$ gives an approximate simulation from $\mu$. We denote the distribution of $Y_a$ by $\mu_a$. We can simulate $Y_{a,+}$ and $Y_{a,-}$ using Algorithm 1 or Algorithm 2. In this paper, we use Algorithm 2, where we simulate from $M_1^{(a)}$ using Algorithm 3. To simulate from $M^{(a)}_2$ we can use Algorithm 4 or Algorithm 5. We select the one that has the larger probability of acceptance. Acceptance probabilities for several choices of the parameters are given in Table \ref{table: accept probs}.

\begin{table}
\begin{center}
\begin{tabular}{c|c|c|c}
$\alpha$ & $a$ &  Algorithm 4   & Algorithm 5  \\ \hline
       & $0.5$     & $0.1739$    & $0.7568$\\
$0.25$ & $10^{-1}$ & $0.3473$    & $0.4521$\\             
       & $10^{-2}$ & $0.5780$    & $0.1338$\\
       & $10^{-4}$ & $0.8098$    & $0.0059$\\ \hline
       & $0.5$     & $0.3671$    & $0.8284$\\
$0.50$ & $10^{-1}$ & $0.5745$    & $0.5798$\\             
       & $10^{-2}$ & $0.7697$    & $0.2457$\\
       & $10^{-4}$ & $0.8738$    & $0.0279$\\ \hline
       & $0.5$     & $0.6180$    & $0.9091$\\
$0.75$ & $10^{-1}$ & $0.7681$    & $0.7556$\\             
       & $10^{-2}$ & $0.8718$    & $0.4822$\\
       & $10^{-4}$ & $0.9050$    & $0.1583$
\end{tabular}
\caption{Acceptance probabilities for Algorithms 4 and 5 for several choices of $\alpha$ and $a$. In all cases we take $c_-=c_+ = 1$ and $\ell_-=\ell_+ = 0.5$.}
\label{table: accept probs}
\end{center}
\end{table}

We now illustrate the performance of our methodology.  For simplicity, we focus on the symmetric case. In this study we choose $\alpha \in \{0.25, 0.5, 0.75\}$, $c_-=c_+ = 1$, and $\ell_-=\ell_+ = 0.5$.  We take $a \in \{0.5, 0.1, 0.01, 0.0001\}$ to demonstrate how $\mu_a$ approaches $\mu$. The results of our simulations are given in Figure \ref{pdf_tweedie} and Table \ref{tests_tweedie}.

First, for each choice of the parameters and each value of $a$, we simulate $5000$ observations. In Figure \ref{pdf_tweedie} we give two types of diagnostic plots. First, we compare kernel density estimates (KDE) of our observations from $\mu_a$ with the density of $\mu$. We can see that, as $a$ decreases, the KDE approaches the density of $\mu$. Second, we present qq-plots and again we see that, as $a$ decreases, the qq-plots approach the diagonal line. 

Next, we perform  formal goodness-of-fit testing using the Kolmogorov-Smirnov (KS) and Cramer-Von Mises (CVM) tests. Here, for each choice the parameters and of $a$, we simulate $100$ samples each comprised of $5000$ observations. For each sample we evaluate the $p$-values of the KS and the CVM tests. We then average these together over the $100$ samples. Recall that under the null hypothesis, the $p$-value follows a $U(0,1)$ distribution. Thus, when averaging over many runs, we expect the average of the $p$-values to be close to $0.5$. The results are given in Table \ref{tests_tweedie}. We can see that the means of the $p$-values get closer to $0.5$ as $a$ decreases. Taken together, the goodness-of-fit results given in  Figure \ref{pdf_tweedie} and Table \ref{tests_tweedie} suggest that the approximation works well for $a=10^{-4}$.

\begin{figure}
\begin{center}
\begin{tabular}{cc}
\includegraphics[scale=.22]{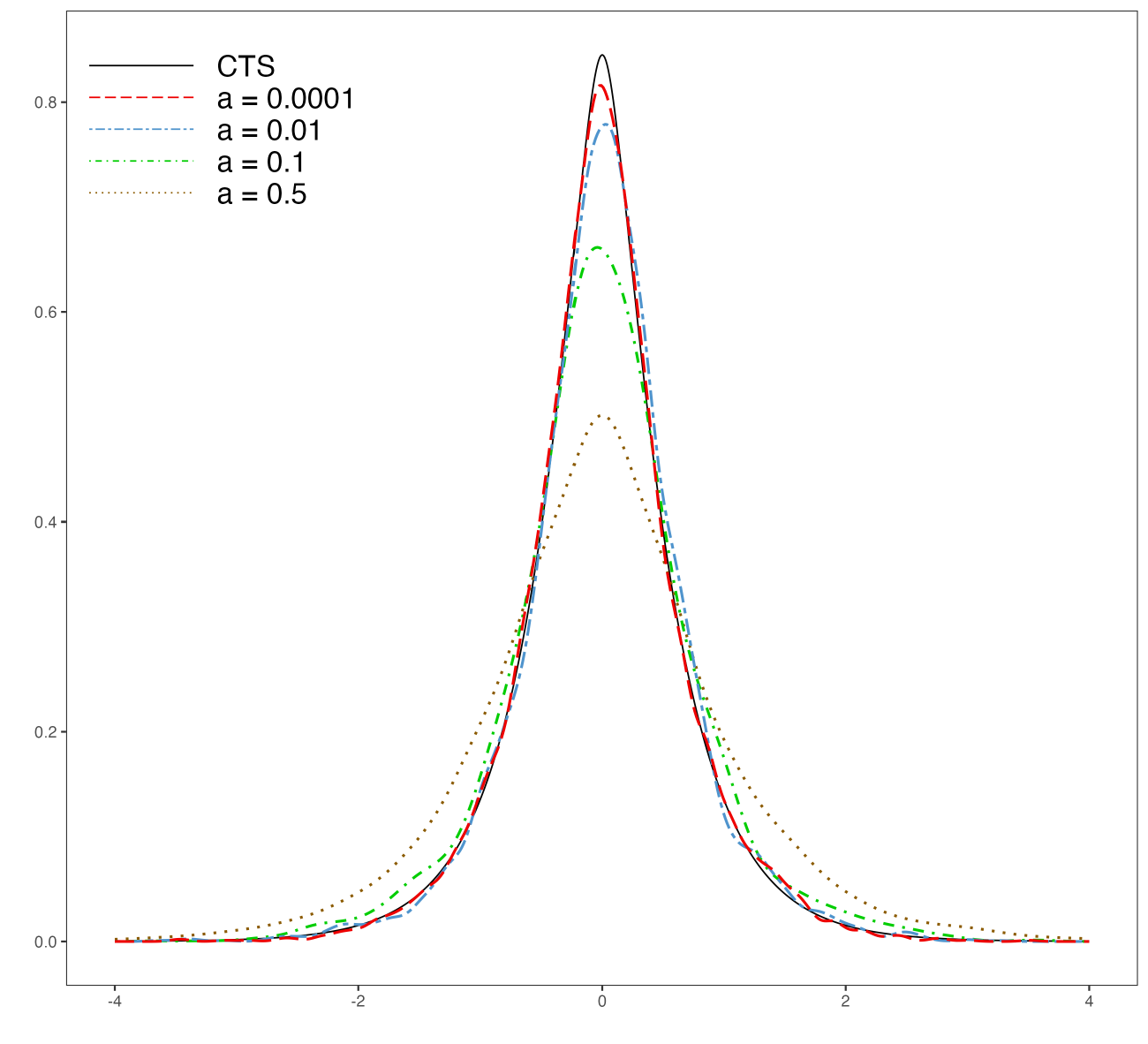} &\includegraphics[scale=.22]{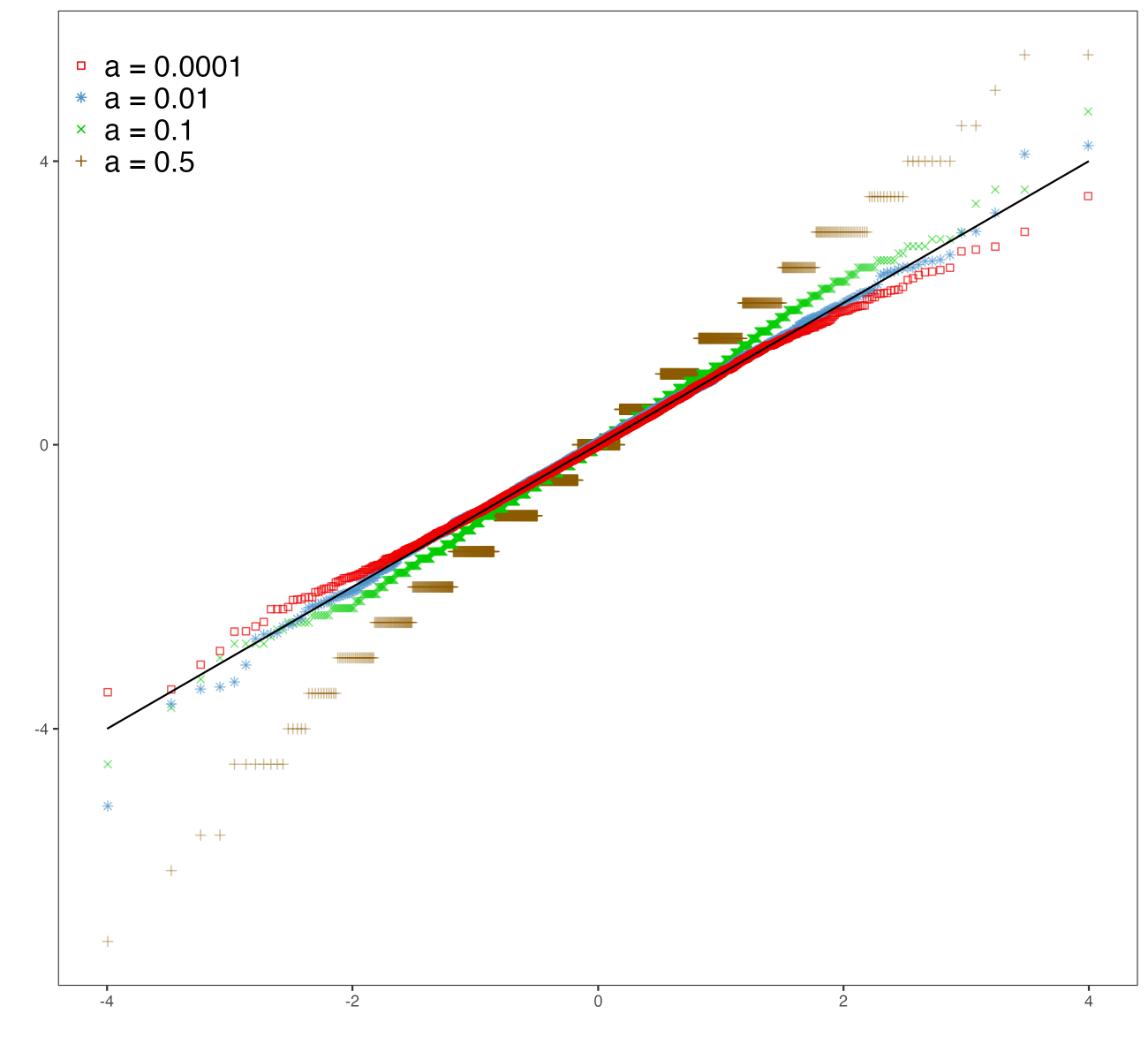}  \\
 \multicolumn{2}{c}{$\alpha = 0.25$}\\
\includegraphics[scale=.22]{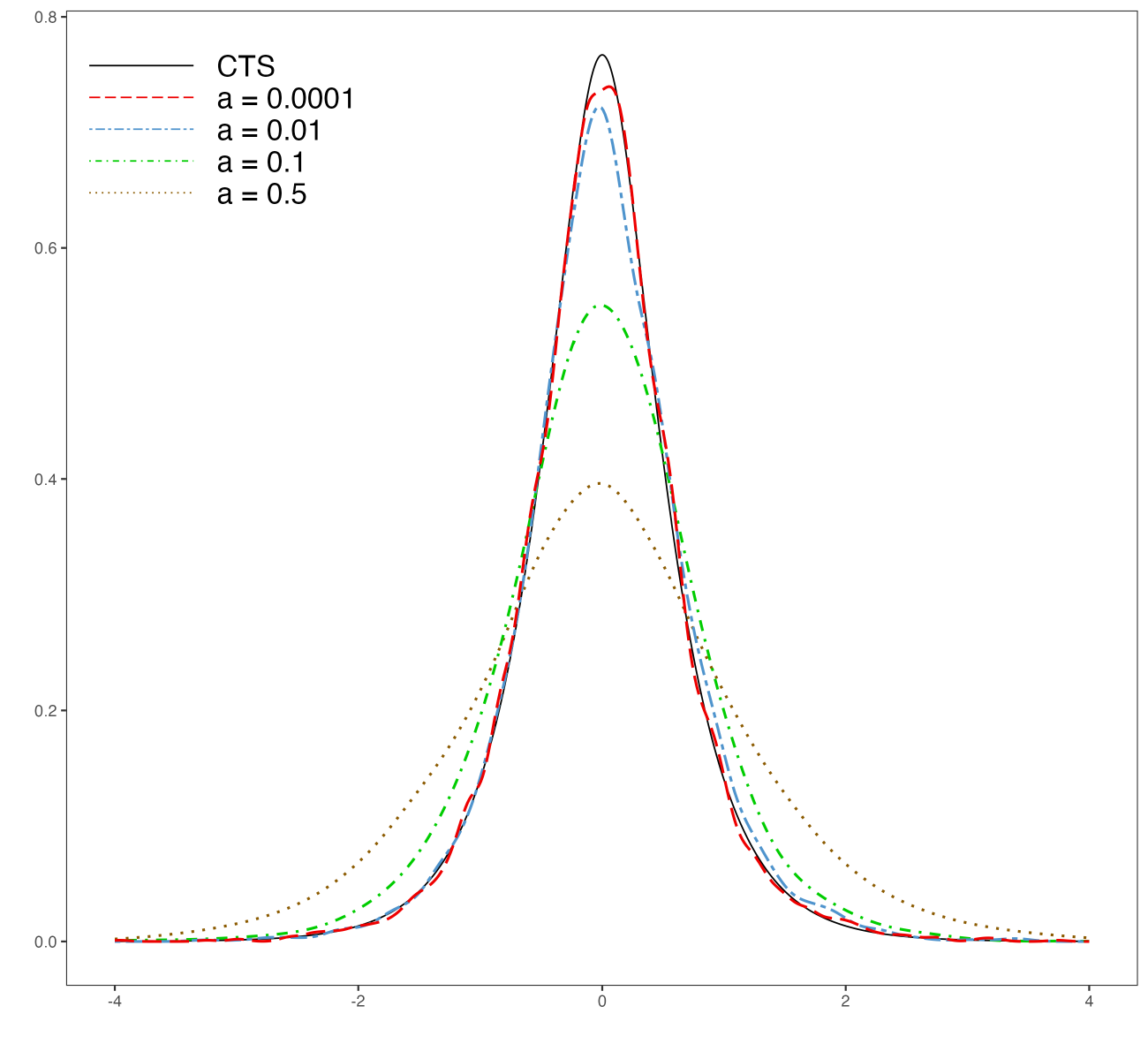} & \includegraphics[scale=.22]{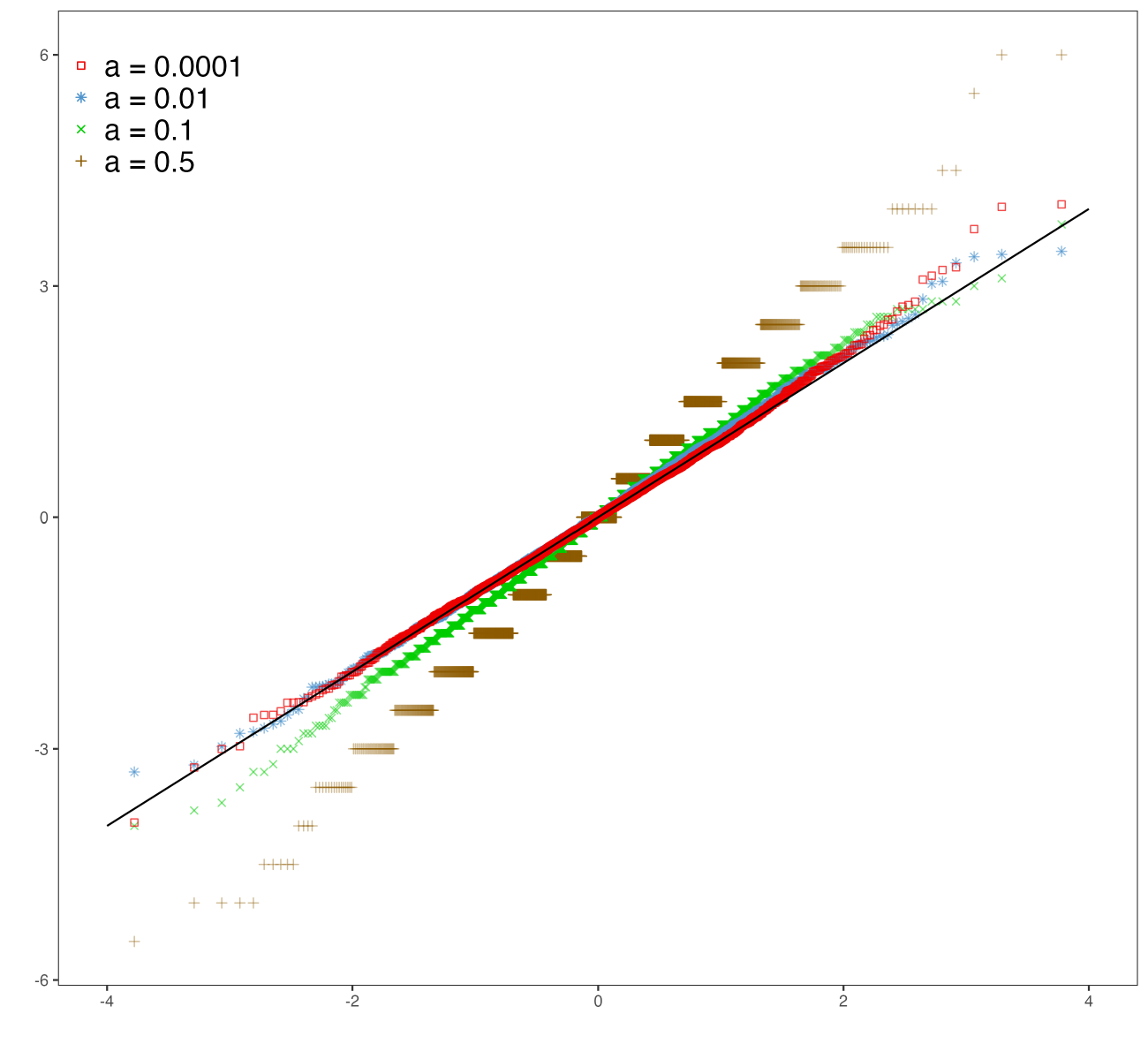}\\
 \multicolumn{2}{c}{$\alpha = 0.50$}\\
\includegraphics[scale=.22]{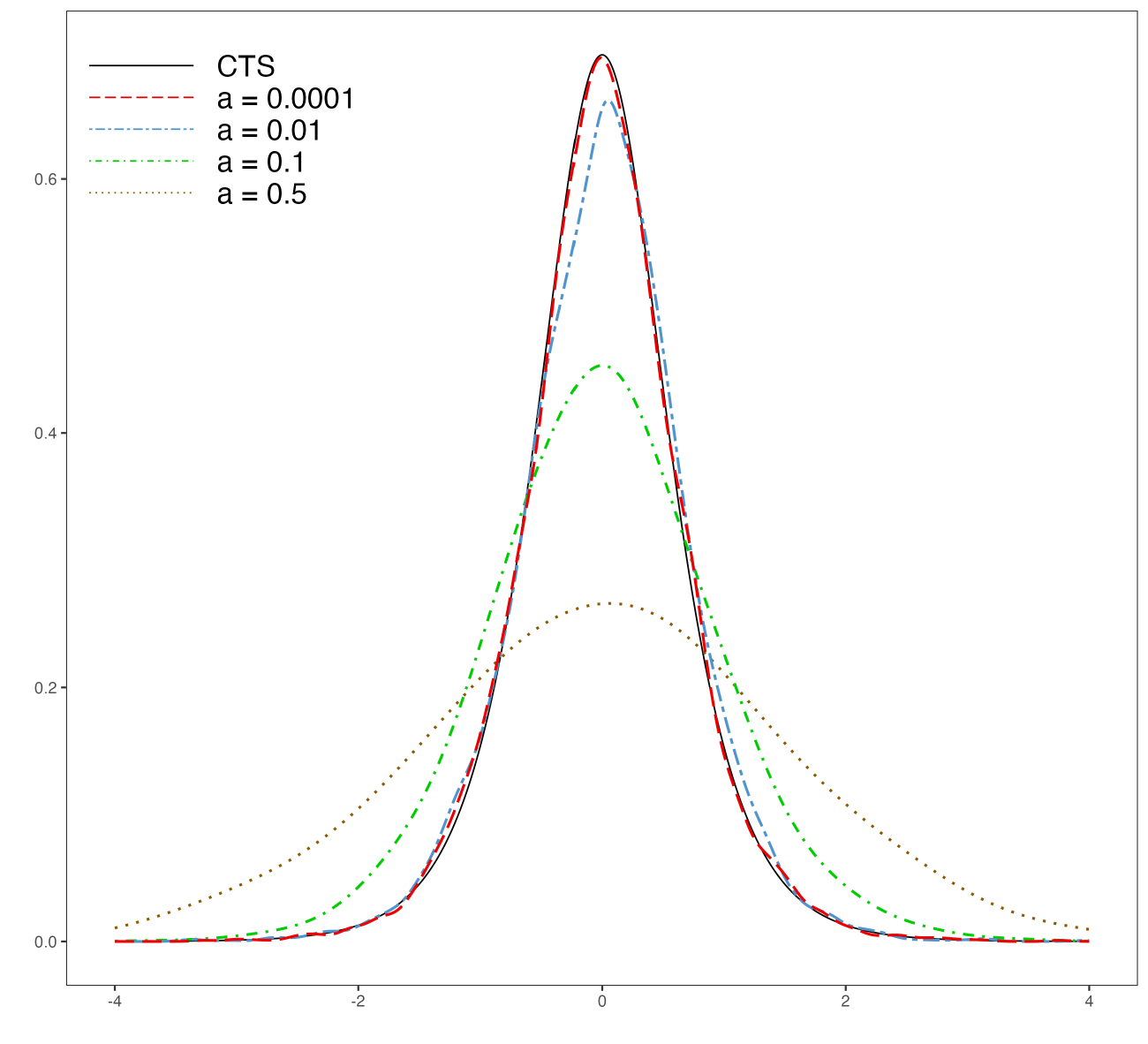} &\includegraphics[scale=.22]{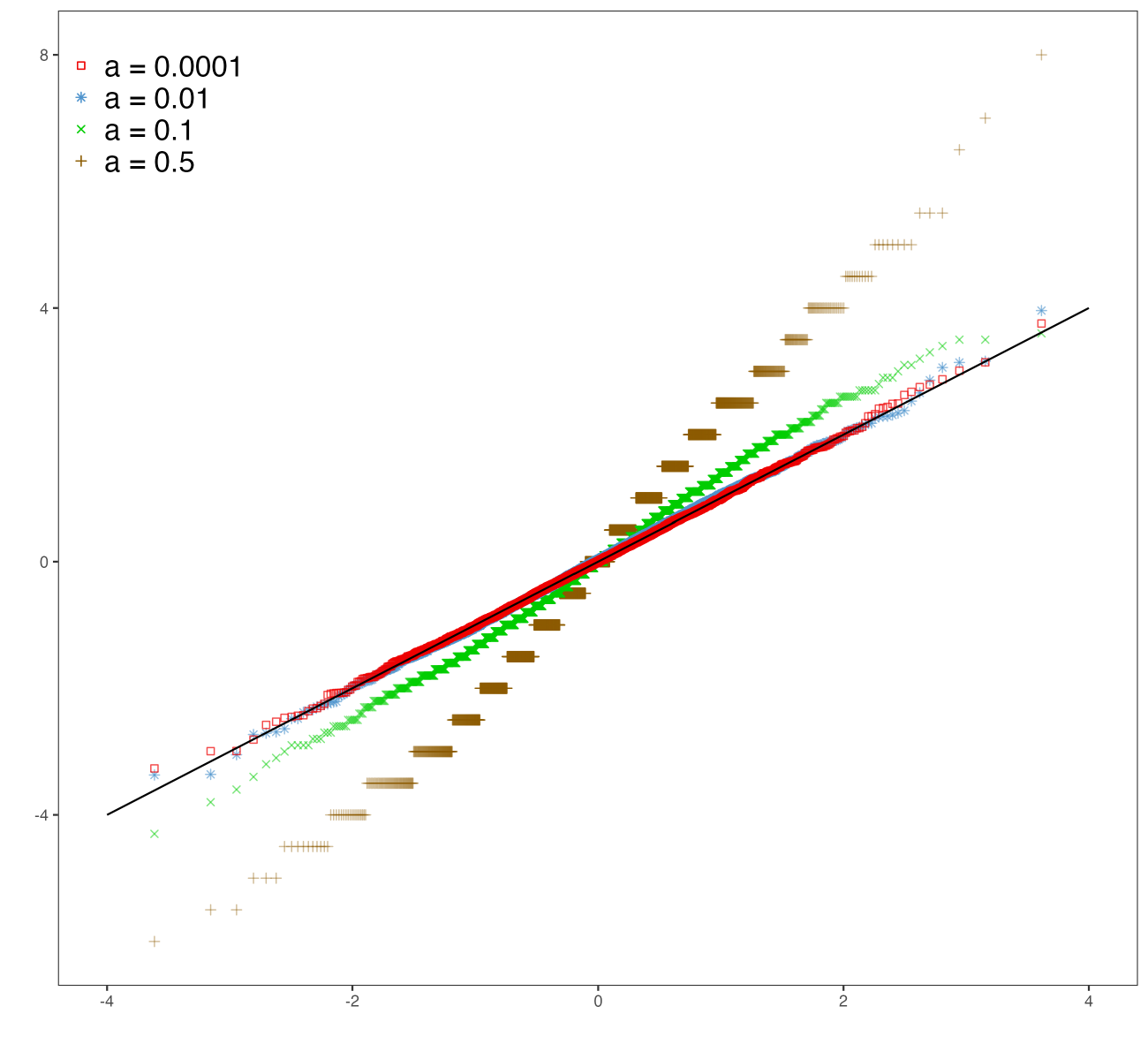} \\
 \multicolumn{2}{c}{$\alpha = 0.75$}
\end{tabular}
\caption{Diagnostic plots for CTS. In the plots on the left, the solid black line is the pdf of the CTS distribution. Overlaid are KDEs based on observations from $\mu_a$ for several choices of $a$. The plots on the right give qq-plots based on samples from $\mu_a$ for several choices of $a$. The solid line is the reference line $y=x$.}
\label{pdf_tweedie}
\end{center}
\end{figure}

\begin{table}
\begin{center}
\begin{tabular}{cc|cc|cc}
$\alpha$ & $a$ & KS, CTS    & CVM, CTS  & KS, PT    & CVM, PT \\ \hline
       & $0.5$     & $0$                 & $0$   & $0$            & $0$\\
$0.25$ & $10^{-1}$ & $\leq 10^{-11}$     & $\leq 10^{-6}$   &  $\leq 10^{-4}$    & $0.0005$\\      
       & $10^{-2}$ & $0.2620$            & $0.4019$ & $0.3797$          & $0.4506$\\
       & $10^{-4}$ & $0.4722$            & $0.4643$ & $0.5142$          & $0.5057$\\ \hline
       & $0.5$     & $0$                 & $0$   & $0$               & $0$\\
$0.50$ & $10^{-1}$ & $\leq 10^{-17}$     & $0$   & $\leq  10^{-5}$   & $\leq 10^{-3}$\\            
       & $10^{-2}$ & $0.2221$            & $0.3773$& $0.3918$          & $0.4664$\\
       & $10^{-4}$ & $0.4644$            & $0.4626$& $0.4805$          & $0.4887$\\ \hline
       & $0.5$     & $0$                 & $0$& $0$               & $0$\\
$0.75$ & $10^{-1}$ & $0$                 & $0$          & $\leq 10^{-11}$   & $\leq 10^{-8}$\\  
       & $10^{-2}$ & $0.0564$            & $0.0800$ & $0.3484$       & $0.4026$\\
       & $10^{-4}$ & $0.5113$            & $0.5125$& $0.4886$       & $0.4853$
\end{tabular} 
\caption{These are p-values for KS and CVM tests for different values of $\alpha$ and $a$ in approximating simulations from CTS and PT distributions. These are averaged over $100$ runs. Under the null hypothesis, the value should be close to $0.5$.}
\label{tests_tweedie}
\end{center}
\end{table}

We now turn to PT distributions. Distribution $\TS(\eta_-,\eta_+,g)$ is PT if
\begin{eqnarray*}
g(x) &=& 1_{[x\le0]} \int_{0}^{\infty} e^{-|x|u} (1+u)^{-\alpha-\ell_--2} u^{\ell_-} \rd u  \\
&&\quad + 1_{[x\ge0]} \int_{0}^{\infty} e^{-xu} (1+u)^{-\alpha-\ell_+-2} u^{\ell_+} \rd u,
\end{eqnarray*}
where $\ell_+,\ell_->0$. It is often convenient to take $\eta_-=c_-(\alpha+\ell_-)(\alpha+\ell_-+1)$ and $\eta_+=c_+(\alpha+\ell_+)(\alpha+\ell_++1)$, where $c_-,c_+\ge0$ with $c_-+c_+>0$ are parameters. In this case, we denote this distribution by PT$_\alpha(c_-,\ell_-,c_+,\ell_+)$. These distributions have finite means, but heavier tails than CTS distribution. If $Y\sim$PT$_\alpha(c_-,\ell_-,c_+,\ell_+)$ with $c_-,c_+>0$, then
$$
\rE[|Y|^\beta]<\infty \mbox{ if and only if } \beta<1+\alpha+\min\{\ell_-,\ell_+\}.
$$
For more on these distributions, see \cite{Grabchak:2016} or \cite{Grabchak:2021b}.

Our methodology for approximate simulation from $\mu=$PT$_\alpha(c_-,\ell_-,c_+,\ell_+)$ is as follows. First choose $a>0$ small, then independently simulate $Y_{a,+}\sim\mathrm{PM}(M^+,a)$ and $Y_{a,-}\sim\mathrm{PM}(M^-,a)$, where $M^-(\rd x) = \eta_- g(-x) x^{-1-\alpha}1_{[x>0]}\rd x$ and $M^+(\rd x) =  \eta_+ g(x)x^{-1-\alpha} 1_{[x>0]}\rd x$. The random variable $Y_a = Y_{a,+}-Y_{a,-}$ gives an approximate simulation from $\mu$ and we denote the distribution of $Y_a$ by $\mu_a$. We can simulate $Y_{a,+}$ and $Y_{a,-}$ using Algorithm 1 or Algorithm 2. In this paper, we use Algorithm 2, where we simulate from $M_1^{(a)}$ using Algorithm 3 and from $M^{(a)}_2$ using Algorithm 4. Note that we cannot use Algorithm 5 since \eqref{eq: cond on q} does not hold in this case.

We now illustrate the performance of our methodology for PT distributions. We again focus on the symmetric case. In this study we choose $\alpha \in \{0.25, 0.5, 0.75\}$, $c_-=c_+ = 1$, and $\ell_-=\ell_+ = 1$. We take $a\in \{0.5, 0.1, 0.01, 0.0001\}$ to demonstrate how $\mu_a$ converges to $\mu$. The results of our simulations are given in Figure \ref{pdf_power} and Table \ref{tests_tweedie}.

Again, for each choice of the parameters and each value of $a$, we simulate $5000$ observations. In Figure \ref{pdf_power} we give two types of diagnostic plots. First, we compare the KDE of our observations from $\mu_a$ with the density of $\mu$. We can see that, as $a$ decreases, the KDE approaches the density of $\mu$. Second, we present qq-plots and again we see that, as $a$ decreases, the qq-plots approach the diagonal line. 

Next, we perform formal goodness-of-fit testing. For each choice the parameters and of $a$, we simulate $100$ samples each comprised of $5000$ observations. For each sample we evaluate the $p$-values of the KS and the CVM tests. We then average these together over the $100$ samples. Again, when averaging over many runs, we expect the average of the $p$-values to be close to $0.5$ under the null hypothesis. The results are given in Table \ref{tests_tweedie}. We can see that the means of the $p$-values get closer to $0.5$ as $a$ decreases. Taken together, the goodness-of-fit results given in  Figure \ref{pdf_power} and Table \ref{tests_tweedie} suggest that the approximation works well for $a=10^{-4}$.

\begin{figure}
\begin{center}
\begin{tabular}{cc}
\includegraphics[scale=.22]{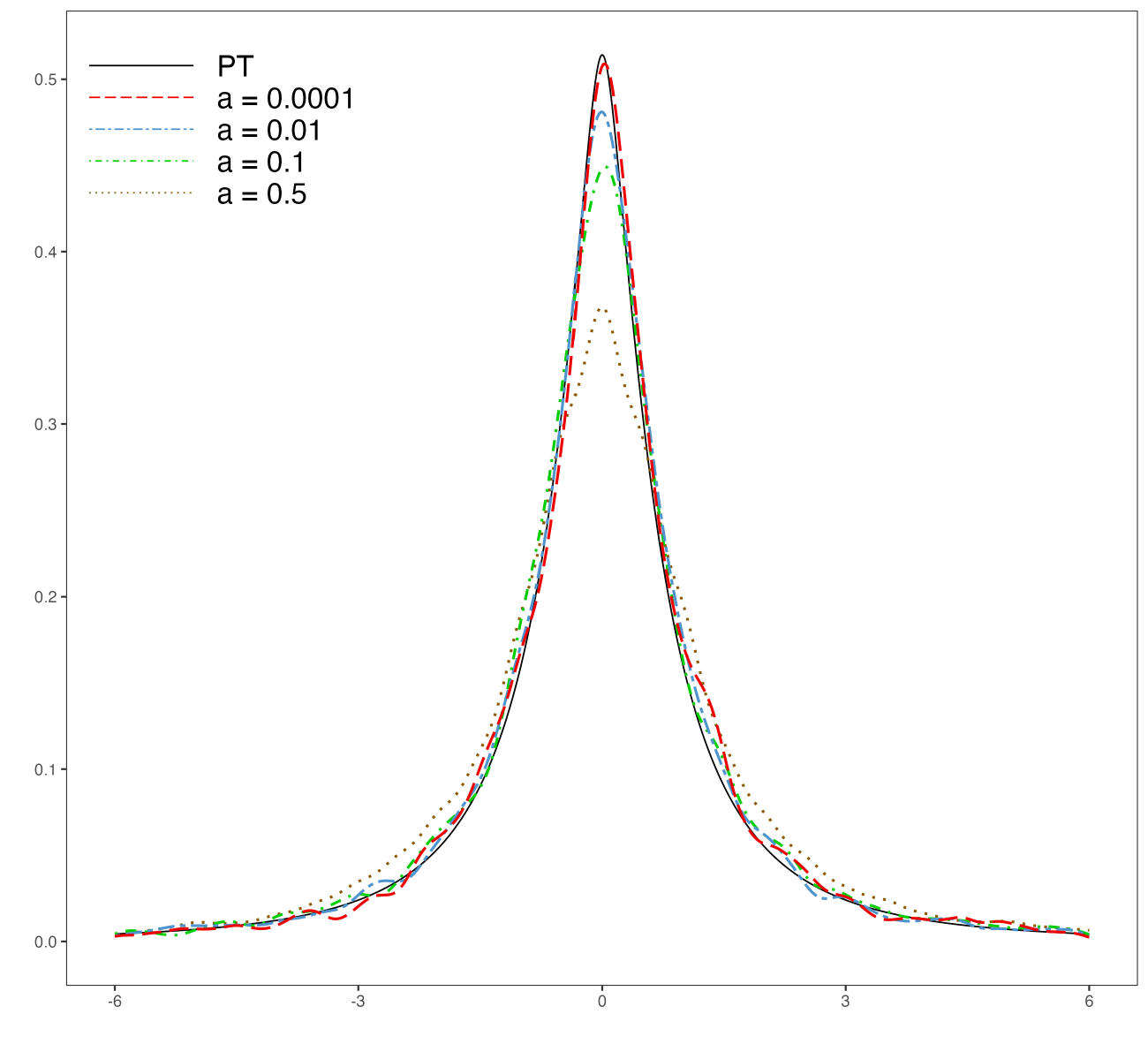} & \includegraphics[scale=.22]{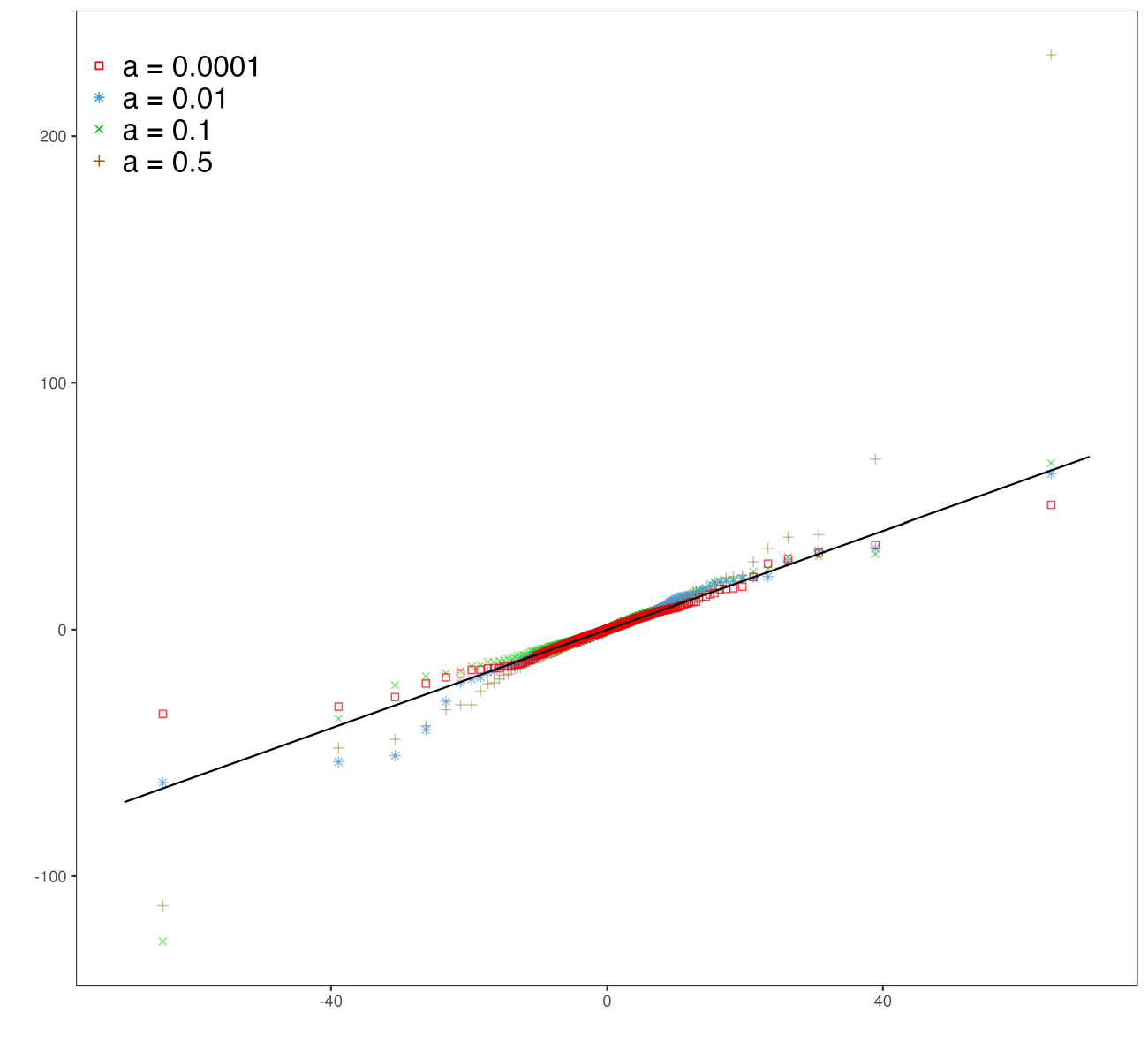} \\
 \multicolumn{2}{c}{$\alpha = 0.25$}\\
\includegraphics[scale=.22]{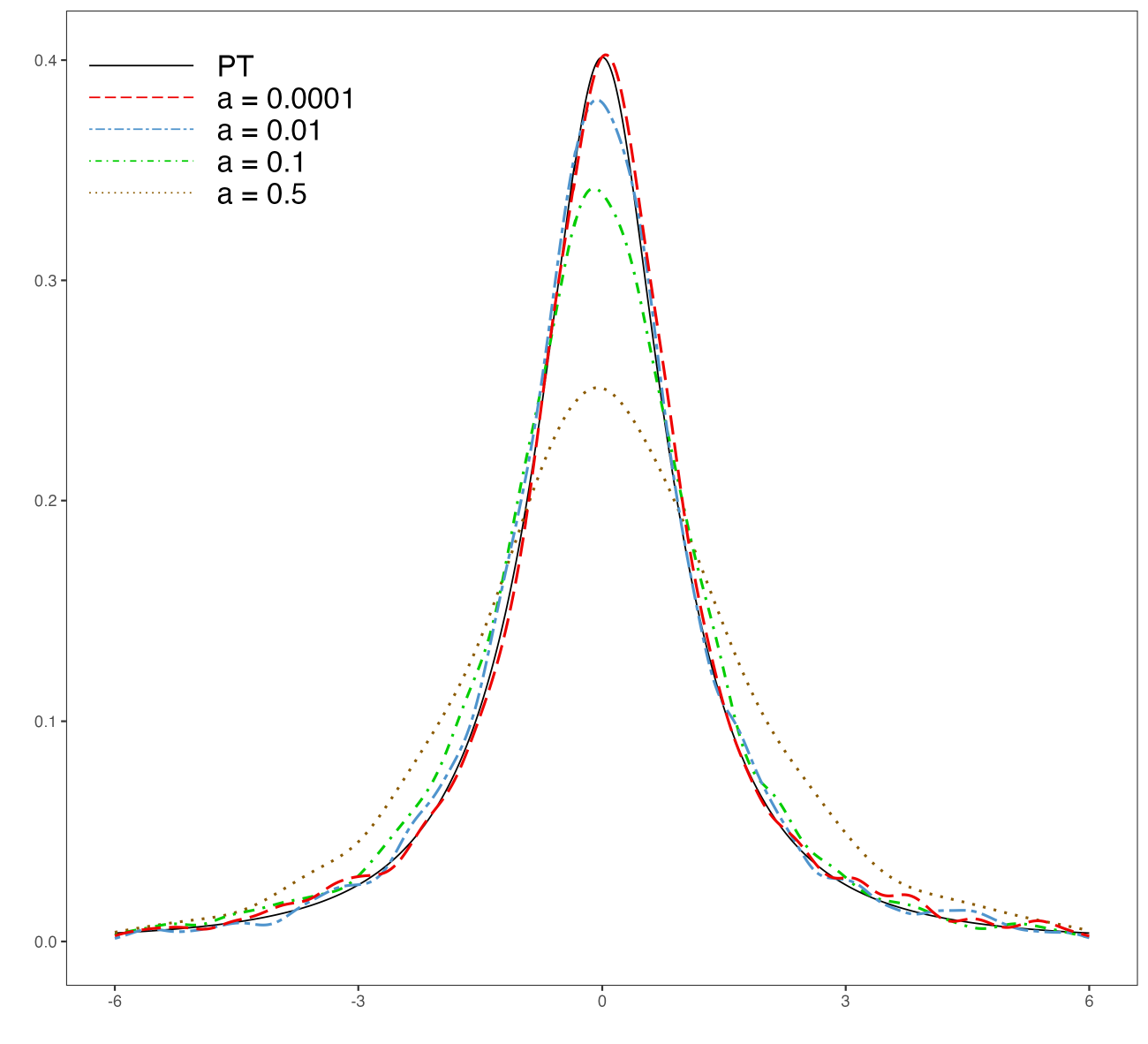} & \includegraphics[scale=.22]{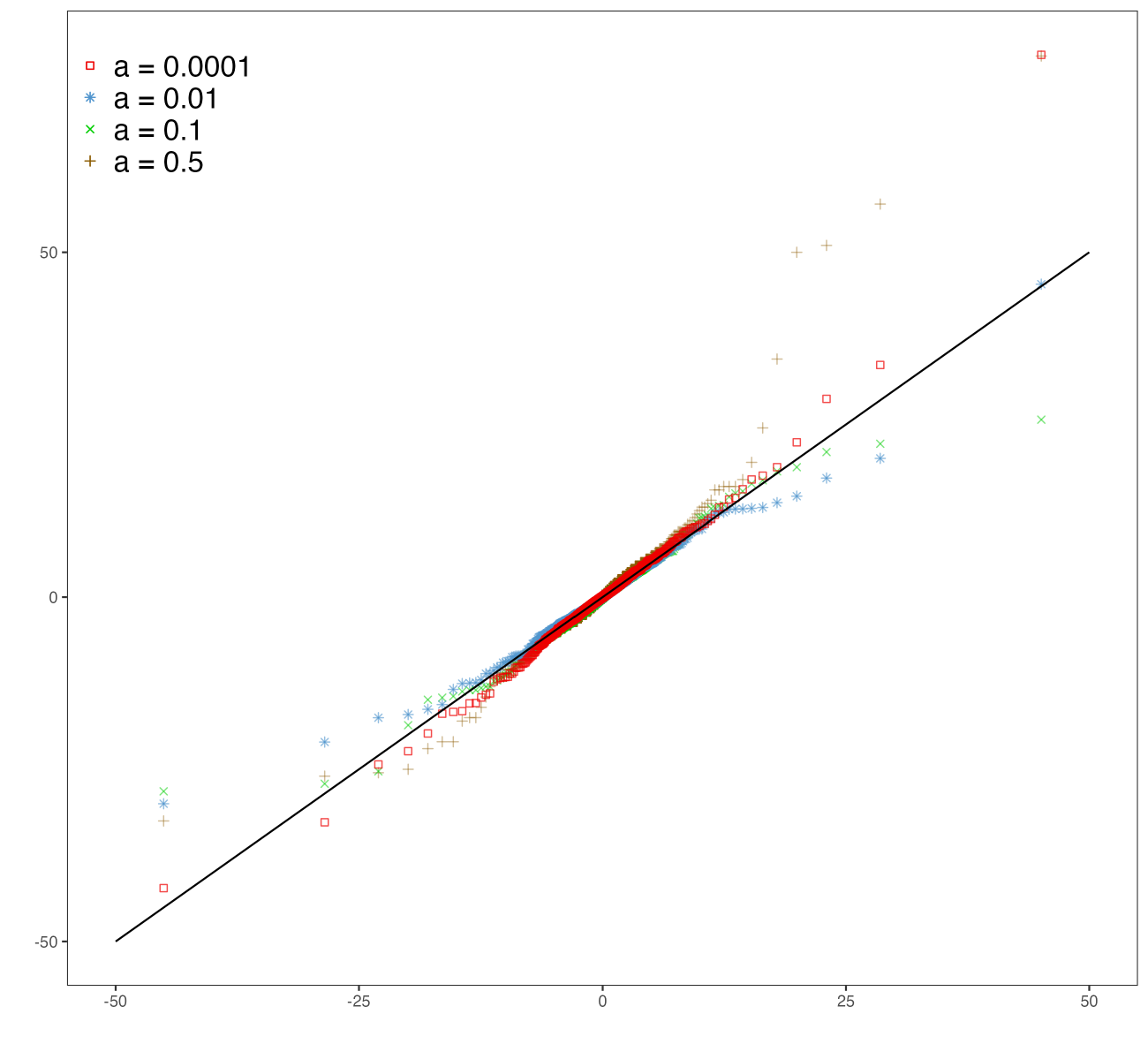} \\
 \multicolumn{2}{c}{$\alpha = 0.50$}\\
\includegraphics[scale=.22]{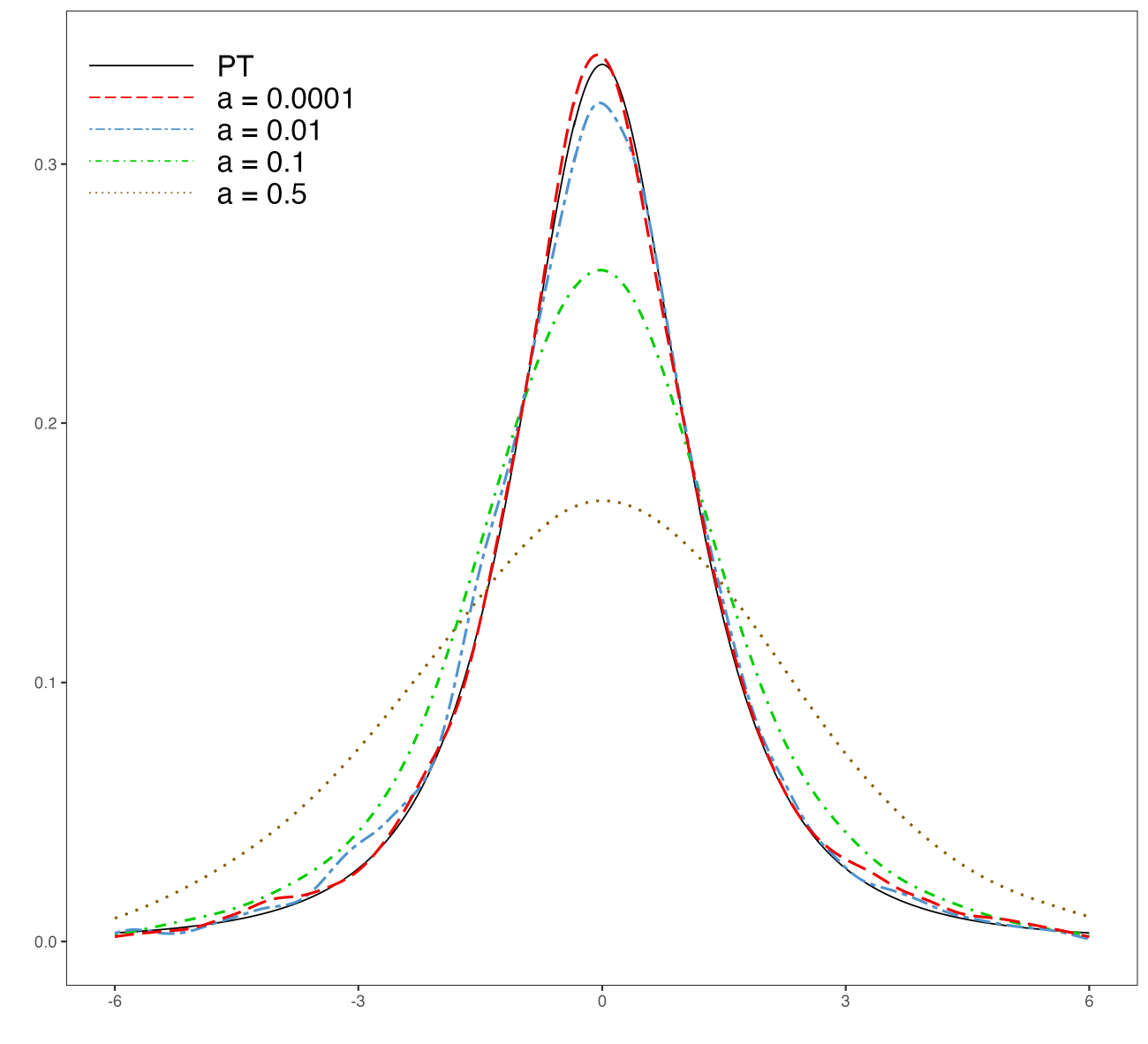} & \includegraphics[scale=.22]{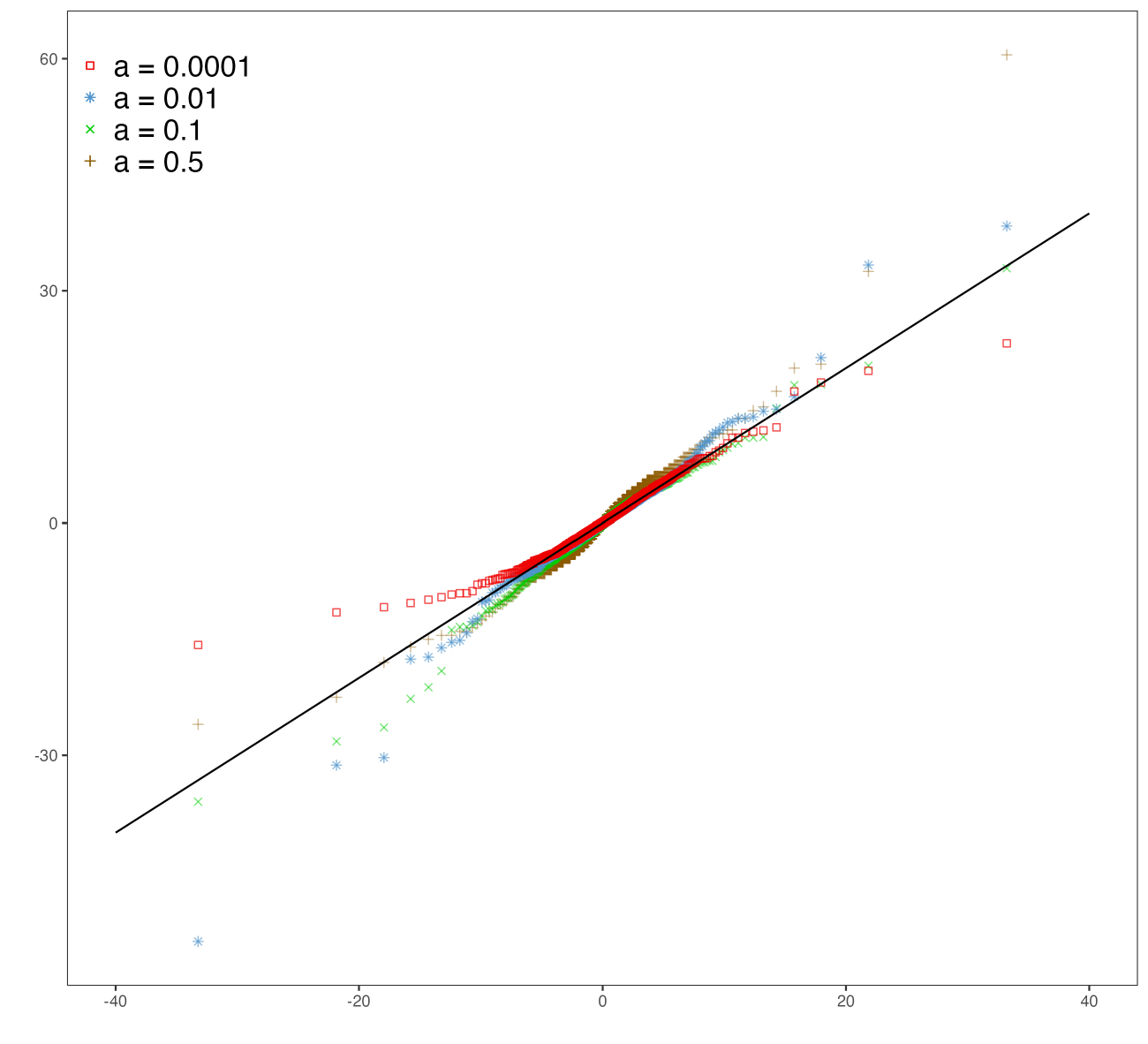} \\
 \multicolumn{2}{c}{$\alpha = 0.75$}
\end{tabular}
\caption{Diagnostic plots for PT. In the plots on the left, the solid black line is the pdf of the PT distribution. Overlaid are KDEs based on observations from $\mu_a$ for several choices of $a$. The plots on the right give qq-plots based on samples from $\mu_a$ for several choices of $a$. The solid line is the reference line $y=x$.}
\label{pdf_power}
\end{center}
\end{figure}

\section{Proofs}\label{sec: proofs}

In this section we prove out theoretical results. We begin with several lemmas.

\begin{lemma}\label{lemma: bounds of exp}
Fix $z\in\mathbb C$.\\
1. We have $\left|e^{z}-1\right| \le |z|e^{|z|}$.\\
2. If $\Re z\le0$, then $\left|e^{z}-1\right|\le 2\wedge|z|$.\\
3. If $\Im z=0$, then for any $n\in\{0,1,2,\dots\}$
$$
\left|e^{iz}-\sum_{k=0}^n \frac{(iz)^k}{k!}\right|\le \min\left\{ \frac{|z|^{n+1}}{(n+1)!}, \frac{2|z|^n}{n!}\right\}.
$$
4. If $z_1,z_2\in\mathbb C$ with $\Re z_1,\Re z_2\le 0$, then
$$
\left|e^{z_1}-e^{z_2}\right| \le 2\wedge |z_1-z_2|.
$$
\end{lemma}

\begin{proof}
We have
\begin{eqnarray*}
\left|e^{z}-1\right|  &\le& \sum_{n=1}^\infty \frac{|z|^n}{n!} 
= |z|\sum_{n=0}^\infty \frac{|z|^{n}}{(n+1)! }\le |z|\sum_{n=0}^\infty \frac{|z|^{n}}{n! }=|z|e^{|z|}.
\end{eqnarray*}
Next, let $z=-a+ib$, where $a=-\Re z$ and $b=\Im z$ and assume that $-a=\Re z\le0$. We have
$$
\left|e^{z}-1\right| = \left|z\int_0^1 e^{tz}\rd t\right|\le |z| \int_0^1 e^{-at}\rd t \le |z|
$$
and, by the triangle inequality, 
\begin{eqnarray*}
\left|e^{z}-1\right|  \le  |e^{-a}|  |e^{ib}| + 1 \le 2.
\end{eqnarray*}
The third part can be found in Section 26 of \cite{Billingsley:1995}. We now turn to the fourth part. Without loss of generality assume that $\Re z_1\le\Re z_2\le 0$ and note that
$$
\left|e^{z_1}-e^{z_2}\right| \le \left|1-e^{z_1-z_2}\right|  \le 2\wedge |z_1-z_2|,
$$
where the final inequality follows by the second part.
\end{proof}

\begin{lemma}\label{lemma: bounds on c nu}
For any $s\in\mathbb R$ we have
$$
\left|C_\nu(s) \right|\le \zeta_1|s| \mbox{ and }
\left|C_\nu(s) -is\gamma\right|\le \frac{1}{2}s^2 \zeta_2.
$$
\end{lemma}

\begin{proof}
By Lemma \ref{lemma: bounds of exp}
$$
\left|C_\nu(s) \right|\le \int_{-\infty}^\infty\left| e^{isz}-1\right|L(\rd z)\le|s| \int_{-\infty}^\infty\left| z\right|L(\rd z) = |s|\zeta_1.
$$
Next, from the definition of $C_\nu$ and $\gamma$, we have
$$
\left|C_\nu(s) -is\gamma\right|\le \int_{-\infty}^\infty \left| e^{isz}-1-isz\right|L(\rd z)\le \frac{1}{2}s^2 \int_{-\infty}^\infty z^2L(\rd z),
$$
where the final inequality follows by Lemma \ref{lemma: bounds of exp}.
\end{proof}

\begin{proof}[Proof of Proposition \ref{prop: conv}]
Our proof is based on showing the convergence of characteristic functions. By l'H\^opital's rule
$$
\lim_{a\downarrow0} \frac{|x|}{a\gamma} (e^{isazx/|x|}-1) = |x| \lim_{a\downarrow0} \frac{\cos(sazx/|x|)-1+i\sin(sazx/|x|)}{a\gamma}  = iszx/\gamma.
$$
By Lemma \ref{lemma: bounds of exp}
\begin{eqnarray*}
 \int_{-\infty}^\infty \frac{|x|}{a\gamma} \left|e^{isazx/|x|}-1\right|L(\rd z) \le |x||s|\int_{-\infty}^\infty |z|L(\rd z)/\gamma<\infty.
\end{eqnarray*}
Thus, by dominated convergence
\begin{eqnarray*}
\lim_{a\downarrow0} \frac{|x|}{a\gamma} C_\nu(sax/|x|) &=& \lim_{a\downarrow0} \int_{-\infty}^\infty \frac{|x|}{a\gamma} \left(e^{isazx/|x|}-1\right)L(\rd z) \\
&=&  ixs\int_{-\infty}^\infty zL(\rd z) /\gamma=ixs.
\end{eqnarray*}
From here
$$
\lim_{a\downarrow0}\hat\mu_a(z) = \hat\mu(z), \ \ \ z\in\mathbb R
$$
by another application of dominated convergence. We just note that by Lemmas \ref{lemma: bounds of exp} and \ref{lemma: bounds on c nu}
\begin{eqnarray*}
\left| e^{\frac{|x|}{a\gamma} C_\nu(sax/|x|) }-1\right| &\le& 2\wedge \frac{|x|}{a\gamma} \left|C_\nu(sax/|x|)\right| \\
&\le& 2\wedge |x||s|\zeta_1/\gamma\\
&\le& \left(2+|s|\zeta_1/\gamma\right) (1\wedge |x|).
\end{eqnarray*}
From here \eqref{eq: levy meas} guarantees that dominated convergence holds.
\end{proof}

We now recall several versions of Esseen’s smoothing lemma, see \cite{Bobkov:2016} for a survey of such results.

\begin{lemma}\label{lemma: Esseen}
For any $T>0$, we have
\begin{eqnarray}\label{eq: Esseen inf}
\left\| F_{a} - F \right\|_\infty \le \frac{1}{\pi}\int_{-T}^T \left|\frac{\hat\mu_{a}(s) - \hat\mu(s)}{s} \right|\rd s+ \frac{12 r_0}{\pi^2 T}
\end{eqnarray}
and for any $T>0$ and any $p\in[1,\infty)$
\begin{eqnarray}\label{eq: Esseen one}
\left\| F_{a} - F \right\|_p^p  &\le& \left\| F_{a} - F \right\|_1 \le \left(\int_{-T}^T\left|\frac{\hat\mu_{a}(s) - \hat\mu(s)}{s}\right|^2\rd s\right)^{1/2} \nonumber\\
&&\quad+ \left(\int_{-T}^T \left|\frac{\hat\mu'_{a}(s) - \hat\mu'(s)}{s}\right|^2\rd s\right)^{1/2}  \nonumber\\
&&\quad+  \left(\int_{-T}^T \left|\frac{\hat\mu_{a}(s) - \hat\mu(s)}{s^2}\right|^2\rd s\right)^{1/2} +\frac{4\pi}{T}.
\end{eqnarray}
Further, for any $p\in[2,\infty)$, if $q = p/(p-1)$, then for any $T>0$
\begin{eqnarray}\label{eq: Esseen p}
\left\| F_{a} - F \right\|_p \le\frac{1}{2}\left(\int_{-T}^T\left|\frac{\hat\mu_{a}(s) - \hat\mu(s)}{s}\right|^q\rd s\right)^{1/q} + \frac{4(p-1)}{T^{1/p}}.
\end{eqnarray}
\end{lemma}

\begin{proof}
Here, \eqref{eq: Esseen inf} is a version of the classical Esseen’s smoothing lemma, see, e.g., page 538 in \cite{Feller:1971} or Theorem 1.5.2 in \cite{Ibragimov:Linnik:1971}. We just use the fact that the essential supremum of the pdf of $\mu$ is upper bounded by $r_0/(2\pi)$. The formula in \eqref{eq: Esseen one} is essentially Corollary 8.3 in \cite{Bobkov:2016}, which it itself  a version of Theorem 1.5.4 in \cite{Ibragimov:Linnik:1971}. The constant $4\pi$ is given in \cite{Ibragimov:Linnik:1971}. To get the bound in \eqref{eq: Esseen one}, we apply the quotient rule and Minkowski's inequality to the bound in \cite{Bobkov:2016}. The formula in \eqref{eq: Esseen p} is given in Corollary 7.2 of \cite{Bobkov:2016}. It is a simple application of the classical Hausdorff-Young inequality, as given in, e.g., Proposition 2.2.16 of \cite{Grafakos:2004}. 
\end{proof}

\begin{lemma}\label{lemma: char func bound}
1. For any $s\in\mathbb R$, we have
\begin{eqnarray*}
\left|\hat\mu_{a}(s) - \hat\mu(s)\right|\le a |\hat\mu(s)|  s^2 e^{0.5 as^2 m_1\zeta_2/\gamma} m_1\zeta_2\frac{1}{2\gamma}
\end{eqnarray*}
and
\begin{eqnarray*}
\left|\hat\mu'_{a}(s) - \hat\mu'(s)\right|\le a|\hat\mu(s)| \zeta_2  \left(\frac{\zeta_1}{2\gamma^2} s^2  m_2 +  \frac{\left|s \right|}{\gamma} m_1 + s^2 e^{0.5 as^2 m_1 \zeta_2/\gamma} m_1^2 \frac{\zeta_1}{2\gamma^2}\right).
\end{eqnarray*}
2. For any $a>0$, $T>0$, and $q\ge1$, we have
\begin{eqnarray}\label{eq: bound on integs inf}
\left(\int_{-T}^T\left|\frac{\hat\mu_{a}(s) - \hat\mu(s)}{s}\right|^q\rd s\right)^{1/q} \le a T e^{0.5aT^2 m_1\zeta_2/\gamma} m_1\frac{\zeta_2}{2\gamma}  r_0^{1/q},
\end{eqnarray}
\begin{eqnarray*}\label{eq: bound on integs q}
\left(\int_{-T}^T\left|\frac{\hat\mu_{a}(s) - \hat\mu(s)}{s^2}\right|^q\rd s\right)^{1/q} \le a e^{0.5aT^2 m_1\zeta_2/\gamma} m_1\frac{\zeta_2}{2\gamma}  r_0^{1/q},
\end{eqnarray*}
and 
\begin{eqnarray*}\label{eq: bound on integs one}
\left(\int_{-T}^T\left|\frac{\hat\mu'_{a}(s) - \hat\mu'(s)}{s}\right|^2\rd s\right)^{1/2} \le a  \frac{\zeta_2 }{\gamma} \left(\frac{\zeta_1}{2\gamma} T m_2 +  m_1 + T e^{0.5aT^2 m_1 \zeta_2/\gamma} m_1^2\frac{\zeta_1}{2\gamma}\right) r_0^{1/2}.
\end{eqnarray*}
\end{lemma}

Our proof of the first part is based on Lemma \ref{lemma: bounds of exp}. A different approach for getting related bounds is used in the proof of Theorem 4.1 in \cite{Arras:Houdre:2019}. However, that approach seems to lead to sub-optimal bounds in this case.

\begin{proof}
By the first part of Lemma \ref{lemma: bounds of exp} and the fact that $|e^{ixs}|=1$
\begin{eqnarray*}
\left|\hat\mu_{a}(s) - \hat\mu(s)\right|&=& |\hat\mu(s)|\left| \exp\left\{\int_{-\infty}^\infty\left(e^{C_\nu(sax/|x|)|x|/(a\gamma)}-e^{isx}\right)M(\rd x) \right\}-1\right|\\
&\le&  |\hat\mu(s)| \exp\left\{\int_{-\infty}^\infty\left|e^{C_\nu(sax/|x|)|x|/(a\gamma)-ixs}-1\right| M(\rd x) \right\}\\
&&\quad\times \int_{-\infty}^\infty\left|e^{C_\nu(sax/|x|)|x|/(a\gamma)-ixs}-1\right|M(\rd x).
\end{eqnarray*}
Combining Lemmas \ref{lemma: bounds of exp} and \ref{lemma: bounds on c nu} with \eqref{eq: real part c is neg} gives
\begin{eqnarray}
&& \int_{-\infty}^\infty\left|e^{C_\nu(sax/|x|)|x|/(a\gamma) -ixs}-1\right| M(\rd x)\nonumber \\
&&\qquad \le \int_{-\infty}^\infty\left|C_\nu(sax/|x|)|x|/(a\gamma) -ixs\right| M(\rd x)\nonumber \\
&&\qquad = \frac{1}{a\gamma}\int_{-\infty}^\infty\left|C_\nu(sax/|x|) -isa\gamma x/|x|\right| |x| M(\rd x)\nonumber \\
&&\qquad \le \frac{a}{2\gamma}s^2\zeta_2\int_{-\infty}^\infty |x| M(\rd x) = \frac{a}{2\gamma} s^2 \zeta_2m_1.\label{eq: bound orig}
\end{eqnarray}
From here the first inequality follows. Turning to the second inequality, let
$$
A_a(s) = \frac{1}{\gamma}\int_{-\infty}^\infty e^{C_\nu(sax/|x|)|x|/(a\gamma) }C'_\nu(sax/|x|)xM(\rd x)
$$
and
$$
A(s) = i\int_{-\infty}^\infty e^{ixs}xM(\rd x).
$$
By dominated convergence, we have $\hat\mu_a'(s) = A_a(s)\hat\mu_a(s)$ and $\hat\mu'(s) = A(s)\hat\mu(s)$. It follows that
$$
\left|\hat\mu'_{a}(s) - \hat\mu'(s)\right| \le \left|A_{a}(s) - A(s)\right||\hat\mu(s)|+|A_a(s)|\left|\hat\mu_{a}(s) - \hat\mu(s)\right|.
$$
By dominated convergence,
$$
C'_\nu(s) = i\int_{-\infty}^\infty e^{izs} zL(\rd z),
$$
and thus, for any $s\in\mathbb R$
$$
|C'_\nu(s)| \le \int_{-\infty}^\infty |z|L(\rd z)=\zeta_1
$$
and
\begin{eqnarray*}
\left|A_a(s) \right| &\le& \frac{\zeta_1}{\gamma}\int_{-\infty}^\infty\left| e^{C_\nu(sax/|x|)|x|/(a\gamma) }\right||x|M(\rd x)  \le \frac{1}{\gamma} m_1 \zeta_1,
\end{eqnarray*}
where the last inequality follows from \eqref{eq: real part c is neg}. Next, using Lemmas \ref{lemma: bounds of exp} and \ref{lemma: bounds on c nu} and \eqref{eq: real part c is neg} gives
\begin{eqnarray*}
\left|A_a(s) -A(s) \right| &\le&  \frac{1}{\gamma}\int_{-\infty}^\infty\left| e^{C_\nu(sax/|x|)|x|/(a\gamma)}\frac{1}{i}C'_\nu(sax/|x|)-\gamma e^{ixs}\right||x|M(\rd x)  \\
 &\le&  \frac{1}{\gamma}\int_{-\infty}^\infty\int_{-\infty}^\infty\left| e^{C_\nu(sax/|x|)|x|/(a\gamma)}e^{izsax/|x|}- e^{ixs}\right||z|L(\rd z)|x|M(\rd x)  \\
 &=&  \frac{1}{\gamma}\int_{-\infty}^\infty\int_{-\infty}^\infty\left| e^{C_\nu(sax/|x|)|x|/(a\gamma)+izsax/|x|-ixs}-1\right||z|L(\rd z)|x|M(\rd x)  \\
 &\le& \frac{1}{\gamma}\ \int_{-\infty}^\infty\int_{-\infty}^\infty\left| C_\nu(sax/|x|)\frac{|x|}{a\gamma}+izsa\frac{x}{|x|}-ixs\right||z|L(\rd z) |x|M(\rd x)  \\
 &\le& \frac{1}{a\gamma^2}  \int_{-\infty}^\infty  \left| C_\nu(sax/|x|)-i\gamma as\frac{x}{|x|}\right| |x|^2M(\rd x) \int_{-\infty}^\infty |z|L(\rd z)\\
 &&\qquad +  \frac{a\left|s \right|}{\gamma}\ \int_{-\infty}^\infty |x|M(\rd x)\int_{-\infty}^\infty|z|^2L(\rd z)  \\
  &\le& \frac{a}{2\gamma^2} s^2 m_2 \zeta_1\zeta_2 +  \frac{a\left|s \right|}{\gamma} m_1\zeta_2.
\end{eqnarray*}
Putting everything together gives the first part. The second part follows immediately from the first, Minkowski's inequality, and the fact that $|\hat\mu(s)|\le1$ for each $s\in\mathbb R$.
\end{proof}

\begin{proof}[Proof of Theorem \ref{thrm: main gen}]
The proof follows by combining Lemmas \ref{lemma: Esseen} and \ref{lemma: char func bound}. We illustrate the approach for the first bound only, as the rest can be verified in a similar way.  By \eqref{eq: Esseen inf} for any $T>0$
\begin{eqnarray*}
\left\| F_{a} - F \right\|_\infty  &\le& \frac{1}{\pi}\int_{-T}^T \left|\frac{\hat\mu_{a}(s) - \hat\mu(s)}{s} \right|\rd s+ \frac{12 r_0}{\pi^2 T}\\
&\le& \frac{1}{\pi}a T e^{0.5aT^2 m_1\zeta_2/\gamma} m_1\frac{\zeta_2}{2\gamma}  r_0 + \frac{12 r_0}{\pi^2 T},
\end{eqnarray*}
where the second line follows by \eqref{eq: bound on integs inf} with $q=1$. Taking $T=a^{-1/2}$ gives the result.
\end{proof}

\begin{proof}[Proof of Theorem \ref{thrm: main TS}]
First, fix $\delta\in(0,1)$ and note that, since $\lim_{x\to0}g(x)=1$, there exists $\beta>0$ such that for any $x\in(-\beta,\beta)$ we have $g(x)\ge1-\delta$. It follows that for $|s|>1$
\begin{eqnarray*}
|\hat\mu(s)| &=& \exp\left\{- \int_{-\infty}^\infty\left(1-\cos(|sx|)\right)M(\rd x)\right\} \\
&=& \exp\left\{-\int_0^\infty\left(1-\cos(|s|x)\right)\left(\eta_+ g(x)+\eta_- g(-x)\right) x^{-1-\alpha}\rd x\right\} \\
&=& \exp\left\{-|s|^\alpha \int_0^\infty\left(1-\cos(x)\right)\left(\eta_+ g(x/|s|)+\eta_- g(-x/|s|)\right)x^{-1-\alpha}\rd x\right\} \\
&\le& \exp\left\{-|s|^\alpha \int_0^\beta \left(1-\cos(x)\right)\left(\eta_+ g(x/|s|)+\eta_- g(-x/|s|)\right)x^{-1-\alpha}\rd x\right\}\\
&\le& \exp\left\{-|s|^\alpha\left(\eta_++\eta_-\right) (1-\delta) \int_0^\beta \left(1-\cos(x)\right)x^{-1-\alpha}\rd x\right\} \\
&=& \exp\left\{-|s|^\alpha A\right\},
\end{eqnarray*}
where
$$
A = (1-\delta)\left(\eta_++\eta_-\right) \int_0^\beta \left(1-\cos(x)\right)x^{-1-\alpha}\rd x>0.
$$
Clearly $|\hat\mu(s)|$ is integrable and, hence, $r_0<\infty$ in this case.

Lemma \ref{lemma: char func bound} implies that for any $q\ge1$
\begin{eqnarray*}
\left(\int_{-T}^T\left|\frac{\hat\mu_a(s)-\hat\mu(s)}{s}\right|^q\rd s\right)^{1/q} &\le& am_1\frac{\zeta_2}{2\gamma}\left( \int_{-T}^T |s|^q |\hat\mu(s)|^q e^{0.5qas^2m_1\zeta_2/\gamma} \rd s\right)^{1/q}  \\
&\le& am_1 \frac{\zeta_2}{\gamma}\left( \left(  \int_{0}^1 s^q e^{qa m_1 \zeta_2/\gamma} \rd s \right)^{1/q} \right.\\
&&\ \left.+ \left( \int_{1}^T s^qe^{-qs^\alpha A} e^{0.5qAs^{\alpha}s^{2-\alpha}am_1\zeta_2/(\gamma A)} \rd s\right)^{1/q} \right) \\
&\le&  am_1 \frac{\zeta_2}{\gamma} \left( e^{a m_1 \zeta_2/\gamma} \right.\\
&&\ +  \left.\left(\int_{1}^\infty s^qe^{-qs^\alpha A } e^{0.5qAs^{\alpha}T^{2-\alpha}am_1\zeta_2/(\gamma A)} \rd s \right)^{1/q}\right),
\end{eqnarray*}
where the second line follows by Minkowski's inequality and symmetry. Taking $q=1$ and combining this with Lemma \ref{lemma: Esseen} gives
\begin{eqnarray*}
\|F-F_a\|_\infty&\le& \inf_{T>0}\left( \frac{m_1 }{\pi}a e^{a m_1 \zeta_2/\gamma} \frac{\zeta_2}{\gamma} \right.\\
&&\quad\left.+  \frac{m_1}{\pi}a \frac{\zeta_2}{\gamma} \int_{1}^\infty s e^{-s^\alpha A } e^{0.5As^{\alpha}T^{2-\alpha}am_1\zeta_2/(\gamma A)} \rd s  +\frac{12 r_0}{\pi^2 T}\right)\\
&\le& a\frac{m_1}{\pi} e^{a m_1 \zeta_2/\gamma}\frac{\zeta_2}{\gamma}  + a\frac{m_1}{\pi}  \frac{\zeta_2}{\gamma} \int_{1}^\infty s  e^{-s^{\alpha}A/2} \rd s\\
&&\quad  + a^{1/(2-\alpha)} \left(\frac{m_1\zeta_2}{A\gamma}\right)^{1/(2-\alpha)} \frac{12 r_0}{\pi^2} = \mathcal O\left( a^{1/(2-\alpha)}\right),
\end{eqnarray*}
where we take $T = a^{-1/(2-\alpha)} \left(\frac{A\gamma}{m_1\zeta_2}\right)^{1/(2-\alpha)}$ in the final inequality.  The other parts of the theorem can be shown in a similar way.
\end{proof}

\begin{proof}[Proof of Proposition \ref{prop: facts mixed Pois}]
The first part follows from the fact that
\begin{eqnarray*}
\sum_{k=1}^{\infty}  \frac{\ell_k^{(a)}}{k!}\left(e^{iza k} -1\right) &=&  \int_{0}^{\infty} e^{-x/a} \sum_{k=1}^{\infty} \frac{1}{k!}\left(e^{iza k} -1\right) (x/a)^{k} M(\rd x)\\
&=&  \int_{0}^{\infty} e^{-x/a} \sum_{k=1}^{\infty} \frac{1}{k!}\left((e^{iza}x/a)^k -(x/a)^{k} \right) M(\rd x)\\
&=&  \int_{0}^{\infty} e^{-x/a} \left(e^{e^{iza}x/a} -e^{x/a} \right) M(\rd x) \\
&=&  \int_{0}^{\infty} \left(e^{(e^{iza}-1)x/a} - 1\right) M(\rd x).
\end{eqnarray*}
The second part follows immediately from the first. We now turn to the third part. We have
\begin{eqnarray*}
p_k^{(a)} &=& P(Y_a =  ak) = P(Z(X/a) =  k) = \frac{1}{k!}\rE\left[e^{-X/a}(X/a)^k\right] \\
&=& \frac{a^{-k}}{k!}\rE\left[e^{-X/a}X^k\right] = \frac{a^{-k}}{k!} (-1)^k \psi^{(k)}(1/a).
\end{eqnarray*}
where $\psi(s) = \rE[e^{-sX}]$ for $s\ge0$ is the Laplace transform of $X$ and $\psi^{(k)}$ is the $k$th derivative of $\psi$. Formal substitution shows that
$$
\psi(s) =  \hat\mu(is)= \exp \left\{ \int_{0}^\infty \left(e^{-xs}-1\right) M(\rd x) \right\}.
$$
From here we see that $p_0^{(a)}=\psi(1/a)=e^{-\ell_+^{(a)}}$. Next, setting $f(s) = \log \psi(s)$ we note that the $k$th derivative of $f$ is given by
$$
f^{(k)}(s)= (-1)^k\int_{0}^\infty e^{-xs} x^k M(\rd x), \ \ k=1,2,\dots.
$$
It follows that $f^{(k)}(1/a) =(-a)^k \ell_k^{(a)}$. From standard results about derivatives of exponential functions, we get
$$
\psi^{(k)}(s) = \sum_{j=1}^{k-1} \binom{k-1}{j} \psi^{(j)}(s)f^{(k-j)}(s) , \ \ s>0 .
$$
see, e.g., Lemma 1 in \cite{Grabchak-DTS}. It follows that
\begin{eqnarray*}
p_k^{(a)} &=& \frac{a^{-k}}{k!} (-1)^k \psi^{(k)}(1/a) = \frac{1}{k!} (-a)^{-k}  \sum_{j=1}^{k-1} \binom{k-1}{j} \psi^{(j)}(1/a)f^{(k-j)}(1/a)\\
&=& \frac{1}{k!} (-a)^{-k}  \sum_{j=1}^{k-1} \binom{k-1}{j} p_j^{(a)}(-a)^j j! (-a)^{k-j} \ell_{k-j}^{(a)}\\
&=& \frac{1}{k}  \sum_{j=1}^{k-1} \frac{1}{(k- j-1)!} p_j^{(a)} \ell_{k-j}^{(a)},
\end{eqnarray*}
as required. The fourth part follows by a standard conditioning argument. 
\end{proof}

\section{Extensions}\label{sec: extensions}

In this section we consider some extensions of our work.  These results are not complete, but aim to suggest directions for future research.

\subsection{Infinite $r_0$}

The results in Theorem \ref{thrm: main gen} assume that $r_0=\int_{-\infty}^\infty |\hat\mu(s)|\rd s<\infty$, which implies that $\mu$ has a bounded density. Here, we briefly discuss what happens if we allow $r_0=\infty$. In this case, we no longer need $|\hat\mu(s)|$ in the bounds in Lemma \ref{lemma: char func bound}. While we can apply $|\hat\mu(s)|\le1$, we can get better bounds as follows. By \eqref{eq: real part c is neg}, the fourth part of Lemma \ref{lemma: bounds of exp}, and \eqref{eq: bound orig}, we have for $s\in\mathbb R$
\begin{eqnarray*}
\left|\hat\mu_{a}(s) - \hat\mu(s)\right|&=&\left| e^{\int_{-\infty}^\infty\left(e^{C_\nu(sax/|x|)|x|/(a\gamma)}-1\right)M(\rd x)}-e^{\int_{-\infty}^\infty\left(e^{isx}-1\right)M(\rd x)}\right|\\
&\le&\int_{-\infty}^\infty\left|e^{C_\nu(sax/|x|)|x|/(a\gamma)}- e^{isx}\right| M(\rd x)\le \frac{a}{2\gamma} s^2 \zeta_2 m_1.
\end{eqnarray*}
From here, by arguments as in the proof of Lemma \ref{lemma: char func bound}, we get for $s\in\mathbb R$
\begin{eqnarray*}
\left|\hat\mu'_{a}(s) - \hat\mu'(s)\right|\le a \zeta_2  \left(\frac{1}{2\gamma^2} s^2 \zeta_1m_2 +  \frac{\left|s \right|}{\gamma} m_1 + s^2  m_1^2\zeta_1\frac{1}{2\gamma^2}\right).
\end{eqnarray*}
It follows that for any $a>0$, $T>0$, and $q\ge1$, we have
\begin{eqnarray}\label{eq: bound on integs inf 2}
\left(\int_{-T}^T\left|\frac{\hat\mu_{a}(s) - \hat\mu(s)}{s}\right|^q\rd s\right)^{1/q} \le aT^{1+1/q} \frac{2^{1/q-1}}{\gamma(q+1)^{1/q} } \zeta_2 m_1,
\end{eqnarray}
\begin{eqnarray*}\label{eq: bound on integs q 2}
\left(\int_{-T}^T\left|\frac{\hat\mu_{a}(s) - \hat\mu(s)}{s^2}\right|^q\rd s\right)^{1/q} \le aT^{1/q} 2^{1/q-1}\zeta_2 m_1,
\end{eqnarray*}
and 
\begin{eqnarray*}\label{eq: bound on integs one 2}
\left(\int_{-T}^T\left|\frac{\hat\mu'_{a}(s) - \hat\mu'(s)}{s}\right|^2\rd s\right)^{1/2} \le a \sqrt 2 \zeta_2  \left(\frac{T^{3/2}}{2\sqrt 3 \gamma^2} \zeta_1m_2 +  \frac{T^{1/2}}{\gamma} m_1 + m_1^2\zeta_1\frac{T^{3/2}}{2\sqrt 3\gamma^2}\right).
\end{eqnarray*}
Now, arguments as in the proof of Theorem \ref{thrm: main gen} give the following.

\begin{prop}
Assume that $m_1,\zeta_1,\zeta_2$ are all finite.  For any $p\in[2,\infty)$ and any $a>0$
$$
\left\| F_{a} - F \right\|_p \le \mathcal O\left(a^{1/(2p)}\right).
$$
If, in addition, $m_2<\infty$, then for any  $p\in[1,\infty)$ and any $a>0$
\begin{eqnarray*}
\left\| F_{a} - F \right\|_p &\le& \mathcal O\left(a^{2/(5p)}\right).
\end{eqnarray*}
\end{prop}

Here, for $p\in[1,2)$, we get a slightly worse bound than in Theorem \ref{thrm: main gen}. We do not know if this due to an actually slower rate of convergence or if it is an artifact of the proof. We cannot give a bound for the $L^\infty$ metric without assuming $r_0<\infty$, as \eqref{eq: Esseen inf} requires this assumption.  However, we can get a bound for the closely related L\'evy metric, which metrizes weak convergence. Specifically, combining \eqref{eq: bound on integs inf 2} with Theorem 3.2 in \cite{Bobkov:2016} leads to the rate $\mathcal O\left(a^{1/3}\log(1/a)\right)$.

\subsection{Normal Variance Mixtures}

In this subsection we give rates of convergence for approximating normal variance mixtures. A normal variance mixture is the distribution of a random variable of the form $\sqrt X W$, where $W\sim N(0,1)$ and $X$ is a positive random variable independent of $W$. Such distributions often arise in financial application, see, e.g.,\ \cite{Sabino:2023}, \cite{Bianchi:Tassinari:2024}, and the references therein.

Throughout this section, let $W\sim N(0,1)$  and let $Z=\{Z(t):t\ge0\}$ be a L\'evy process with $Z(1)\sim \nu=\ID_0(L)$. Assume that  $L$ is symmetric, i.e.,\ that $L(B)=L(-B)$ for each $B\in\mathfrak B(\mathbb R)$. This implies that $\nu$ is also symmetric and that its cumulant generating function $C_\nu$ is real. Assume further that
 \begin{eqnarray}\label{eq: moments of L}
 \int_{-\infty}^\infty\left(|z|\vee z^2\right) L(\rd z)<\infty
 \end{eqnarray}
 and set
 $$
 \gamma^* =  \int_{-\infty}^\infty z^2 L(\rd z).
 $$
Let $X\sim \mu=\ID_0(M)$ with $M((-\infty,0])=0$ and 
$$
m_1 := \int_0^\infty x M(\rd x)<\infty.
$$
Assume that $X$ is independent of $W$ and $Z$ when appropriate. Let $\mu^*$ be the distribution of $\sqrt X W$ and let $\mu_a^*$ be the distribution of $\sqrt{\frac{a}{\gamma^*}} Z(X/a)$. By a conditioning argument, it is easily checked that for any $a>0$ the characteristic function of $\mu^*_{a}$ is
$$
\hat \mu^*_a(s) = \exp\left\{\int_0^\infty \left(e^{C_\nu(s\sqrt {a/\gamma^*})x/a}-1\right)M(\rd x)\right\}  \quad s\in\mathbb R
$$
and the characteristic function of $\mu^*$ is
$$
\hat\mu^*(s) = \exp\left\{\int_0^\infty \left(e^{-s^2x/2}-1\right)M(\rd x)\right\}, \quad s\in\mathbb R.
$$

\begin{prop}
We have
$$
\sqrt{\frac{a}{\gamma^*}} Z(X/a)  \cond \sqrt X W \mbox { as } a\downarrow0.
$$
\end{prop}

\begin{proof}
First note that the fact that $L$ is symmetric and satisfies \eqref{eq: moments of L} implies that
\begin{eqnarray}
C_\nu(s\sqrt {a/\gamma^*}) &=& \int_{-\infty}^{\infty}\left(e^{izs\sqrt {a/\gamma^*}}-1\right) L(\rd z) \nonumber \\
&=& \int_{-\infty}^{\infty}\left(e^{izs\sqrt {a/\gamma^*} }-1-izs \sqrt \frac{ a}{\gamma^*} \right) L(\rd z).\label{eq: C nu sym}
\end{eqnarray}
Now, by Lemma \ref{lemma: bounds of exp}
\begin{eqnarray}\label{eq: bound in Gaus case}
\left|\frac{1}{a} \left(e^{izs\sqrt {a/\gamma^*}}-1-izs \sqrt\frac{a}{\gamma^*} \right) \right|\le \frac{z^2s^2 }{2 \gamma^*} 
\end{eqnarray}
and
$$
\frac{1}{a} \left|e^{izs\sqrt {a/\gamma^*}}-1-izs \sqrt \frac{ a}{\gamma^*} + \frac{a}{2 \gamma^*}z^2s^2 \right| \le \frac{\sqrt a}{6}  \left|zs\right|^3 (\gamma^*)^{-3/2}
\to0 \mbox{ as } a\downarrow0.
$$
It follows that 
$$
\frac{1}{a} \left(e^{izs\sqrt {a/\gamma^*}}-1-izs \sqrt \frac{ a}{\gamma^*}  \right) =  -\frac{z^2s^2 }{2\gamma^*} +\frac{1}{a} \left(e^{izs\sqrt {a/\gamma^*} }-1-izs \sqrt \frac{a}{\gamma^*}  + \frac{az^2s^2}{2\gamma^*} \right)
$$
and by dominated convergence
\begin{eqnarray*}
\lim_{a\downarrow0} \frac{1}{a} C_\nu(s\sqrt {a/\gamma^*}) &=& \lim_{a\downarrow0}  \frac{1}{a} \int_{-\infty}^{\infty}\left(e^{izs\sqrt{ a/\gamma^*}}-1-izs  \sqrt \frac{ a}{\gamma^*} \right) L(\rd z)\\
&=& -\frac{1}{2 \gamma^*}s^2 \int_{-\infty}^{\infty} z^2 L(\rd z) = -\frac{1}{2}s^2.
\end{eqnarray*}
From here we get $\lim_{a\downarrow0}\hat \mu^*_a(s) = \hat \mu^*(s)$ for each $s\in\mathbb R$ using another application of dominated convergence. To see that it holds note that by \eqref{eq: real part c is neg}, Lemma \ref{lemma: bounds of exp}, and \eqref{eq: bound in Gaus case}, we have
\begin{eqnarray*}
\left|e^{C_\nu(s\sqrt {a/\gamma^*})x/a}-1\right| \le \left|C_\nu(s\sqrt {a/\gamma^*})\frac{x}{a}\right| \le \frac{s^2 }{2 \gamma^*}\int_{-\infty}^\infty z^2 L(\rd z) |x|,
\end{eqnarray*}
which is integrable with respect to $M$.
\end{proof}

Henceforth, assume that 
$$
\zeta_3 :=  \int_{-\infty}^\infty |z|^3 L(\rd z)<\infty \mbox{ and } r_0 := \int_{-\infty}^\infty \left|\hat\mu^*(s)\right|\rd s  <\infty.
$$

\begin{thrm}\label{thrm: main var mix}
Let $F_a$ be the cdf of $\mu_a^*$ and let $F$ be the cdf of $\mu^*$. Assume that $r_0,m_1,\zeta_3$ are all finite.  For any $a>0$, 
$$
\left\| F_{a} - F \right\|_\infty \le  a^{1/6}\left(  \frac{1}{6\pi} \left( \frac{1}{\sqrt {\gamma^*}}\right)^3 e^{\left(\gamma^*\right)^{-3/2} \zeta_3m_1/6} \zeta_3m_1 + \frac{12}{\pi^2} \right)  r_0= \mathcal O\left(a^{1/6}\right)
$$
and for any $p\in[2,\infty)$ and any $a>0$
\begin{eqnarray*}
\left\| F_{a} - F \right\|_p &\le& \left(  \frac{ a^{1/6}}{12} \left(\gamma^*\right)^{-3/2}  e^{\left(\gamma^*\right)^{-3/2} \zeta_3m_1/6} \zeta_3m_1 r_0^{1-1/p} + 4(p-1) a^{1/(6p)} \right) \\
&=& \mathcal O\left(a^{1/(6p)}\right).
\end{eqnarray*}
\end{thrm}

For technical reasons, we are not able to get a rate of convergence for $p\in[1,2)$. 

\begin{proof}
By arguments similar to those in the proof of Lemma \ref{lemma: char func bound}, we have
\begin{eqnarray*}
\left|\hat\mu^*_{a}(s) - \hat\mu^*(s)\right| &\le&  |\hat\mu^*(s)| \exp\left\{\int_{0}^\infty\left|e^{C_\nu(s\sqrt {a/\gamma^*})x/a}-e^{-s^2x/2}\right| M(\rd x) \right\}\\
&&\quad\times \int_{0}^\infty\left| e^{C_\nu(s\sqrt {a/\gamma^*})x/a}-e^{-s^2x/2}\right|M(\rd x).
\end{eqnarray*}
By the fourth part of Lemma \ref{lemma: bounds of exp} and \eqref{eq: C nu sym}, we have
\begin{eqnarray*}
&& \int_{0}^\infty\left| e^{C_\nu(s\sqrt {a/\gamma^*})x/a}-e^{-s^2x/2}\right|M(\rd x)\nonumber \\
&&\qquad \le \int_{0}^\infty\left| C_\nu(s\sqrt {a/\gamma^*}) \frac{x}{a}+ \frac{s^2x \gamma^*}{2\gamma^*} \right|M(\rd x) \nonumber \\
&&\qquad \le  \int_{0}^\infty \frac{x}{a} \int_{-\infty}^\infty \left|e^{izs\sqrt {a/\gamma^*} }-1-izs \sqrt \frac{ a}{\gamma^*} + \frac{s^2z^2a}{2\gamma^*} \right| L(\rd z) M(\rd x) \nonumber \\
&&\qquad \le \frac{\sqrt a}{6}\left|\frac{s}{\sqrt {\gamma^*}}\right|^3\int_{-\infty}^\infty |z|^3 L(\rd z)\int_{0}^\infty xM(\rd x) = \frac{\sqrt a}{6} \left|\frac{s}{\sqrt {\gamma^*}}\right|^3 \zeta_3m_1.\label{eq: bound for norm}
\end{eqnarray*}
It follows that, for any $a>0$, $T>0$, and $q\ge1$ we have
\begin{eqnarray*}\label{eq: bound on integs inf normal}
\left(\int_{-T}^T\left|\frac{\hat\mu^*_{a}(s) - \hat\mu^*(s)}{s}\right|^q\rd s\right)^{1/q} \le \frac{\sqrt a}{6} T^2 \left( \frac{1}{\sqrt {\gamma^*}}\right)^3  e^{\frac{\sqrt a}{6} \left( \frac{T}{\sqrt {\gamma^*}}\right)^3 \zeta_3m_1} \zeta_3m_1 r_0^{1/q}.
\end{eqnarray*}
From here we proceed as in the proof of Theorem \ref{thrm: main gen}.
\end{proof}

\end{document}